\documentclass[12pt,a4paper,twoside]{scrartcl}
\usepackage[utf8]{inputenc}
\usepackage[english]{babel}

\usepackage{amsfonts}
\usepackage{amssymb}
\usepackage{amsthm}

\usepackage{mathtools}

\usepackage{dynkin-diagrams}

\usepackage{dsfont}
\usepackage{yfonts} 
\usepackage{upgreek} 
\usepackage{ytableau}
\usepackage{color, colortbl}

\usepackage[left=3cm,right=3cm,top=3cm,bottom=3cm]{geometry}

\usepackage{imakeidx}
\makeindex[columns=2, title=Index, intoc]

\usepackage[refpage, intoc, english]{nomencl}
\makenomenclature

\usepackage{fancyhdr}
\usepackage{framed,color}
\usepackage{ifthen} 
\usepackage{mathabx}
\usepackage{mathrsfs}
\usepackage[shortlabels]{enumitem} 

\usepackage{tikz}
\usetikzlibrary{arrows}
\usetikzlibrary[topaths]
\usepackage{cancel}

\usepackage{tikz-cd}
\usetikzlibrary{cd}

\usepackage{hyperref}
\usepackage{babelbib}
\usetikzlibrary{babel}
\usepackage[babel=true]{csquotes}

\usepackage{picture} 
\usepackage{cleveref}
\usepackage[title]{appendix}
\makeatletter
\def\smallunderbrace#1{\mathop{\vtop{\m@th\ialign{##\crcr
				$\hfil\displaystyle{#1}\hfil$\crcr
				\noalign{\kern3\p@\nointerlineskip}%
				\tiny\upbracefill\crcr\noalign{\kern3\p@}}}}\limits}
\makeatother

\let\saveLongrightarrow\Longrightarrow
\makeatletter
\renewcommand*{\Longrightarrow}{%
	\mathrel{\rlap{\fontfamily{cmrx}\fontencoding{OT1}\selectfont=}%
		\hphantom{\saveLongrightarrow}%
		\llap{$\m@th\Rightarrow$}}}
\makeatother

\let\geq\geqslant
\let\leq\leqslant

\usepackage{amsthm}

\newcommand{\cl}{\mathsf{cl}}

\newcommand{\uq}{\mathbf U_q(\mathfrak g)}

\newcommand{\qfa}{R_q(G)}

\newcommand{\Q}{\mathbb{Q}}
\newcommand{\Z}{\mathbb{Z}}
\newcommand{\N}{\mathbb{N}}
\newcommand{\C}{\mathbb{C}}
\newcommand{\F}{\mathbb{F}}

\newcommand{\A}{\mathbf{A}}

\theoremstyle{definition}
\newtheorem{defi/}{Definition}[section]
\newenvironment{defi}
{%
	\pushQED{\qed}\begin{defi/}}
	{\popQED\end{defi/}}
\newtheorem{ex/}[defi/]{Example}
\newenvironment{ex}
{%
	\pushQED{\qed}\begin{ex/}}
	{\popQED\end{ex/}}
\theoremstyle{plain}
\newtheorem{prop}[defi/]{Proposition}
\newtheorem{thm}[defi/]{Theorem}

\newtheorem{cor}[defi/]{Corollary}
\newtheorem{lem}[defi/]{Lemma}
\theoremstyle{remark}
\newtheorem{rem}[defi/]{Remark} 

\theoremstyle{definition}
\newtheorem{theorem}{Theorem}

\usepackage{braket} 

\makeatletter
\def\@tvsp{\mathchoice{{}\mkern-4.5mu}{{}\mkern-4.5mu}{{}\mkern-2.5mu}{}}
\def\ltrivert{\left|\@tvsp\left|\@tvsp\left|}
\def\rtrivert{\right|\@tvsp\right|\@tvsp\right|}
\def\lbivert{\left|\@tvsp\left|} 
\def\rbivert{\right|\@tvsp\right|} 
\def\llangle{\left\langle\@tvsp\left\langle}
\def\rrangle{\right\rangle\@tvsp\right\rangle}
\makeatother

\newcommand{\I}{\mathbf{I}}
\newcommand{\uqb}{\mathbf{B}_{\boldsymbol{c}}}
\newcommand{\uqbl}{\mathbf{B}^\omega_{\boldsymbol{c},\boldsymbol{s}}}
\newcommand{\uqdis}{\mathbf{B}_{\boldsymbol{c}^\diamondsuit}}
\newcommand{\uqbmin}{\mathbf{B}_{\boldsymbol{-c}}}
\newcommand{\uqbs}{\mathbf{B}_{\boldsymbol{c,s}}}
\newcommand{\uqbt}{\mathbf{B}_{\boldsymbol{d,t}}}
\newcommand{\uqbschi}{\mathbf{B}^\chi_{\boldsymbol{c}}}
\newcommand{\qfaa}{R_q(\mathcal A)}
\newcommand{\qfaaws}{R_q(\mathcal A)^{W^\Sigma}}
\newcommand{\uqd}{\mathbf{B}_{\boldsymbol{d}}}
\newcommand{\uqds}{\mathbf{B}_{\boldsymbol{d,t}}}
\newcommand{\uqa}{\mathbf{B}_{\boldsymbol{\Phi_{\boldsymbol{a}}(c)},\boldsymbol{\Phi_{\boldsymbol{a}}(s)}}}
\newcommand{\uqat}{\mathbf{B}_{\boldsymbol{\Phi_{\boldsymbol{a}}(d)},\boldsymbol{\Phi_{\boldsymbol{a}}(t)}}}

\newcommand{\comp}{\mathscr{U}}

\newcommand{\Br}{\mathrm{Br}}
\newcommand{\Hom}{\mathrm{Hom}}
\newcommand{\End}{\mathrm{End}}
\newcommand{\Res}{\mathrm{Res}}
\newcommand{\qfachi}{\qfa^{\uqb}_\chi}
\newcommand{\qfachb}{^{\uqb^\chi}\qfa^{\uqb}_\chi}

\catcode`,\active

\catcode`\,12

\catcode`,\active

\catcode`\,12

\newcommand{\Addresses}{{
		\bigskip
		\footnotesize
		\textsc{IMAPP, Radboud Universiteit, P.O. Box 9010, 6500 GL Nijmegen, The Netherlands}\par\nopagebreak
		\textit{E-mail address}: \texttt{stein.meereboer@ru.nl}
		
}}

\newcount\mycount

\newcommand{\Veranstaltung}{Algebraic Topology}
\newcommand{\Aufgabennr}{1}

\makeatletter 
\fancypagestyle{plain}{
	\fancyhf{} 
	\fancyfoot[LE,RO]{\thepage} 
	}

\makeatother 

\numberwithin{table}{section} 

\definecolor{shadecolor}{rgb}{0.8,0.8,0.8}

\newsavebox{\foobox}

\makeatletter
\newcommand{\thickbar}{\mathpalette\@thickbar}
\newcommand{\@thickbar}[2]{{#1\mkern1.5mu\vbox{
			\sbox\z@{$#1\mkern-.5mu#2\mkern-.5mu$}%
			\sbox\tw@{$#1\overline{#2}$}%
			\dimen@=\dimexpr\ht\tw@-\ht\z@-.8\p@\relax
			\hrule\@height.8\p@ 
			\vskip\dimen@
			\box\z@}\mkern1.5mu}
}
\makeatother

\newcommand{\myrule} [3] []{
	\begin{center}
		\begin{tikzpicture}
			\draw[#2-#3, ultra thick, #1] (0,0) to (0.65\linewidth,0);
		\end{tikzpicture}
	\end{center}
}

\newcommand{\demph}[2][]{\emph{#2}\ifthenelse{\equal{#1}{}}{\index{#2|hfill}}{\index{#1|hfill}}} 

\makeatletter
\def\namedlabel#1#2{\begingroup
	#2%
	\def\@currentlabel{#2}%
	\phantomsection\label{#1}\endgroup
}
\makeatother


\usepackage{pdfpages}
\usepackage{comment}

\title{Symmetries for spherical functions of type $\chi$ for quantum symmetric pairs}
\author{ Stein Meereboer}

\begin{document}
\thispagestyle{empty}
\makeatletter
\begin{center}
{\huge \bfseries \textsf \@title}\\[6pt]
{\large\@author}\\
{\Addresses}
\end{center}
\makeatother
\begin{abstract}
	$\textsc{Abstract.}$ Let $\mathfrak{g}$ be a complex semisimple Lie algebra and let $\uq$ denote the associated Drinfel'd Jimbo quantized enveloping algebra. In this paper we study spherical functions of $\uq$ related to characters. We show invariance under the Wang-Zhang braid group operators and show relative Weyl group invariance, when restricted to the quantum torus.
\end{abstract}
\tableofcontents
\numberwithin{equation}{section}
\section{Introduction}
\subsection{Background}
There is a close connection between spherical functions on compact symmetric spaces and orthogonal polynomials. Our interest lies in studying quantum analogs of these objects in the framework of Gail Letzter's quantum symmetric pairs $(\uq,\uqbs)$ \cite{Letzter1999}. \\
One of the early motivations for studying quantum analogs of symmetric pairs was the study of quantum analogs of zonal spherical functions. Pioneering this study was Koornwinder's paper \cite{Koornwinder1993}, where zonal spherical functions related to quantum analogs of the symmetric pair $(SU(2),U(1))$ are interpreted as Askey-Wilson polynomials, which are a special class of Macdonald polynomials \cite{Macdonald2003}. Later, Letzter gave interpretations of zonal spherical functions on quantum symmetric pairs, as Macdonald polynomials. To be precise, when the root system is reduced, the restrictions of zonal spherical functions to a Cartan subalgebra equal Macdonald polynomials \cite{Letzter2004}. Over two decades later, the theory of quantum symmetric pairs has flourished. In particular, insights by Huanchen Bao and Weiqiang Wang have contributed substantially to the understanding of the structure of quantum symmetric pairs with respect to various aspects of the theory. In particular, the theory allows for various tools for studying harmonic analysis of these spaces. Although the general theory of quantum symmetric pairs has come a far way, the understanding of generalized spherical functions is less well developed. In this paper we apply the recent advancements to the study of $\chi$-spherical functions. \\
Classically, $\chi$-spherical functions related to characters of symmetric pairs are invariant under the action of the relative Weyl group. In the quantum setting, relative Weyl group symmetries have only been systematically studied for zonal spherical functions. In \cite{Letzter2003}, Gail Letzter shows that for each coideal subalgebra $\uqbs$ there exists a unique coideal subalgebra $\uqds$ such that the zonal spherical functions related to $(\uqbs,\uqds)$ are invariant under the  relative Weyl group when restricted to the Cartan subalgebra $\uq^0$. These results indicate that there should be invariance properties of $\chi$-spherical functions related to characters of quantum symmetric pairs.\\
Central to this paper are the recent developments by Weiqiang Wang and Weinan Zhang \cite{Wang2023}, where relative braid group symmetries on the coideal subalgebras $\uqbs$ are constructed. These symmetries do not preserve $\uq$ in general, but they preserve the quantized function algebra $\qfa$.
\subsection{The basic idea}
Let $\mathfrak{g}$ be a complex semisimple Lie algebra, and let $\Theta:\mathfrak{g}\to \mathfrak{g}$ be an involutive automorphism. We write $\mathfrak{k}=\{x\in \mathfrak{g}: \Theta(x)=x\}$ to denote the fixed point algebra of $\Theta$ and refer to $(\mathfrak{g},\mathfrak{k})$ as a symmetric pair. Recall that involutive automorphisms $\Theta$, up to conjugation, are determined by Satake diagrams $(\I,\I_\bullet,\tau)$. The theory of quantum symmetric pairs developed in \cite{Letzter1999} provides quantum group analogs of the pair of universal enveloping algebras $(\mathbf{U}(\mathfrak{g}),\mathbf{U}(\mathfrak{k}))$. A quantum symmetric pair $(\uq,\uqbs)$ consists of a Drinfel'd Jimbo quantum group $\uq$ and a coideal subalgebra $\uqbs\subset \uq$, that depends on a family of scalars $\boldsymbol{c},\boldsymbol{s}\in (\F^\times)^{\I\setminus \I_\bullet}$. Such a family of scalars is said to be balanced if $\boldsymbol{s}=0$ and $c_i=c_{\tau(i)}$ for each $i\in\I_\circ=\I\setminus \I_\bullet$ and we put $\mathbf{B}_{\mathbf{c,0}}=\uqb$.
Recent advancements show the existence of a relative braid group action $\Br(W^\Sigma)\curvearrowright \uqbs$ \cite{Wang2023}.
For balanced parameters, the relative braid group operators admit a factorization
$$\boldsymbol{T}_{i,-1}'(b)=\Upsilon_i \mathcal T_{i,-1}' (b)\Upsilon_i^{-1}\qquad\text{for}\qquad b\in \uqb,i\in \I_\circ$$
where $ \mathcal T_{i,-1}' $ is a re-scaled Lusztig braid group operator  \cite[\mbox{Thm 9.9}]{Wang2023} and $\Upsilon_i$ is the rank one quasi-$K$ matrix related to the rank one Satake subdiagram $\{i,\tau(i)\}\cup \I_\bullet$. 
As the rank one quasi-$K$ matrices act locally finitely on finite dimensional $\uq$ weight modules, we obtain a precomposition operator on the quantized function algebra $-\circ \boldsymbol{T}_{i,-1}': \qfa\to\qfa$ defined by $\varphi\mapsto \varphi\circ\boldsymbol{T}_{i,-1}'$. The study of this operator through the representation theory of $\uqb$ will be key in understanding the braid group symmetries on spherical functions related to characters.
\subsection{Results}
In this subsection, we present the main results of the paper.\\
\\
The first goal of the paper is to establish symmetries for characters of quantum symmetric pairs. Through the study of one-dimensional $\uqb$ modules and a case-by-case check, we establish the invariance of characters by the Wang-Zang braid group operators.
\begin{theorem}[Corollary \ref{cor:invariance}]
	Let $\boldsymbol{c}$ be a balanced parameter. Then for each character $\chi\in \widehat{\uqb}$ that occurs in a finite dimensional $\uq$ weight module and each $w\in \Br(W^\Sigma) $, it holds that $\chi\circ \boldsymbol{T}'_{w,-1}=\chi.$
\end{theorem}
Let $\uqbs,\uqds\subset \uq$ be quantum symmetric pair coideal subalgebras, and let $\chi$ and $\eta$ be characters of $\uqbs$ and $\uqds$, respectively. We say that $\varphi\in \qfa$ is a $(\chi,\eta)$-spherical function if $\varphi$ transforms as
$$\varphi(bxb')=\chi(b)\varphi(x)\eta(b')\qquad\text{for all }  b\in \uqbs,b'\in \uqds,\text{ and } x\in \uq.$$
We write $\qfa_\chi^{\uqbs}$ to denote the space of $(\chi,\chi)$-spherical functions for $(\uqbs,\uqbs)$. Through studying the structure of $\chi$-spherical functions and the factorization of relative braid group operators on $\uq$-modules, we show that $\chi$-spherical functions are invariant under the Wang-Zhang braid group operators.
\begin{theorem}[Proposition \ref{prop:WZ}]	Let $\boldsymbol{c}$ be a balanced parameter. Then for each $\varphi \in \qfachi$ and each $w\in \Br(W^\Sigma)$, it holds that $\varphi\circ \boldsymbol{T}_{w,-1}'=\varphi$.
\end{theorem}
The main theorem of the paper generalizes Letzter's relative Weyl group invariance \cite[\mbox{Thm 5.3}]{Letzter2003} to general characters of the coideal subalgebra $\uqb$. We recall the necessary notions to state the invariance. Let $\mathcal A$ be the $\uq$-subalgebra generated by the $K_h$, where $h\in Y_\Theta=\{h\in Y: \Theta(h)=-h\}$, we refer to $\mathcal A$ as the quantum torus. We write $\Res:\qfa\to \Hom(\mathcal A,\F)$ to denote the restriction map to $\mathcal A$ and denote $\qfaa$ by the image of $\Res$. Recall that there is a natural action of the relative Weyl group $W^\Sigma$ on $\Hom(\mathcal A,\F)$. Let $\uqb$ be a quantum symmetric pair coideal subalgebra, and let $\chi\in \widehat{\uqb}$ be a character. Then we show the existence of a coideal subalgebra $\uqb^\chi$ and a character $\chi'\in \widehat{\uqb^\chi}$ such that, when restricted to $\mathcal A$, the $(\chi',\chi)$-spherical functions are relative Weyl group invariant.
We write $\qfachb$ to denote the space of $(\chi',\chi)$-spherical functions.
\begin{theorem}\label{thm:C}[Theorem \ref{thm:mainthm}]
	Let $\boldsymbol{c}$ be a balanced parameter. Then restriction to the quantum torus defines a $\F$-linear map
	$$\mathrm{Res}: \bigoplus_{\chi\in \widehat{\uqb}}{\qfachb}\to \qfaaws.$$
\end{theorem}
For $\chi=\epsilon$ one recovers Gail Letzter's relative Weyl group invariance for balanced parameters \cite[\mbox{Thm 5.3}]{Letzter2003}. Note that our proof of Gail Letzter's result is different from the original proof  \cite[\mbox{Thm 5.3}]{Letzter2003}. 
In Section \ref{sec:limit}, we address the problem of the existence of characters for the coideal subalgebra $\uqbs$. Recall that a symmetric pair $(\mathfrak{g},\mathfrak{k})$ is called irreducible if $\mathfrak{g}$ cannot be written as the
direct sum of two semisimple Lie subalgebras which both admit $\Theta$ as an involution. Furthermore, recall that an irreducible symmetric pair $(\mathfrak{g},\mathfrak{k})$ is said to be of Hermitian type if the Lie algebra $\mathfrak{k}$ has a nontrivial center.
\begin{theorem}[\ref{cor:herm2}]
	Let $(\boldsymbol{c},\boldsymbol{s})$ be a specializable parameter and let the symmetric pair $(\mathfrak{g},\mathfrak{k})$ be irreducible. Then the following three statements are equivalent:
\begin{enumerate}[(i)]
	\item There exist countably infinite many pairwise non-isomorphic one-dimensional $\uqbs$ modules that occur in finite dimensional $\uq$ weight modules.
	\item There exist non-trivial one-dimensional $\uqbs$ modules that occur in finite dimensional $\uq$ weight modules.
	\item The symmetric pair $(\mathfrak{g},\mathfrak{k})$ is of Hermitian type.
\end{enumerate}
\end{theorem}
Consequently, for an irreducible symmetric pair $(\mathfrak{g},\mathfrak{k})$ the space $\bigoplus\limits_{\chi\in \widehat{\uqb}}{\qfachb}$ decomposes into countably infinite many nontrivial constituents and if and only if $(\mathfrak{g},\mathfrak{k})$ is of Hermitian type. Otherwise, the only constituent arises from to the trivial representation. In particular, in the irreducible Hermitian case this ensures that Theorem \ref{thm:C} is nontrivial.
\subsection{Organization}
Sections 2-3 are of preparatory nature. In Section \ref{sec:quantumgrouips}, we fix notation and conventions on quantum groups and the quantized function algebra. Section \ref{sec:iota} introduces quantum symmetric pairs, quasi-$K$ matrices and the Wang-Zhang braid group operators, which we extend to operators on $\qfa$.
The main new results of the paper are contained in Section 4 and 5.
In Section \ref{sec:char}, we study the characters of $\uqbs$ that occur in irreducible left and right $\uq$ weight modules. We establish that the characters of $\uqbs$ occur multiplicity free and investigate their embedding. We show that characters of quantum symmetric pairs are invariant under the Wang-Zhang braid group operators. In Section \ref{sec:sph}, we study spherical functions in some generality. We show the existence and uniqueness of $\chi$-spherical functions. In Section \ref{sec:limit} we characterize the existence of one-dimensional $\uqbs$ modules occurring in $\uq$ highest weight modules. Section \ref{sec:barqfa} is the heart of the paper and here we combine the results of the previous sections to show that spherical functions related to characters are invariant under the Wang-Zhang braid group operators. The rest of the paper is devoted to the interpretation of these symmetries when restricted to the quantum torus.
Appendix \ref{Appendix} is devoted to completing the proof of Proposition \ref{prop:charinv} and appendix \ref{apenB} contains the proof of Lemma \ref{lem:lusztiglong}. Both proofs involve case-by-case checks. 
\begin{center}
	$\textsc{Acknowledgment}$
\end{center}
The author would like to thank Erik Koelink for his patience, guidance and for many valuable comments. The author would also like to thank Stefan Kolb for the hospitality, discussions and many valuable comments. The author is grateful to a referee for a careful reading and numerous suggestions and corrections.
\begin{center}
	$\textsc{Funding}$
\end{center}
My research was funded by NWO grant \texttt{OCENW.M20.108}. This article is based upon work from COST Action CaLISTA \texttt{CA21109} supported by COST (European Cooperation in Science and Technology).
\section{Quantum groups and bar involutions}\label{sec:quantumgrouips}
\subsection{Quantum groups}\label{subsec:qg}
We introduce notation as in \cite{Lusztig2010}. Let $\mathfrak{g}$ be be a complex semisimple Lie algebra with  Cartan matrix $C=(c_{ij})_{i,j\in \I}$. Let $D=\text{diag}(d_i\,:\,d_i\in \Z_{\geq 1}, \,i\in\I)$ be a symmetrizer, meaning that $DC$ is symmetric and $\gcd\{d_i\,:\, i\in\I\}=1$. Fix a set of simple roots $\Pi=\{\alpha_i\,:\,i\in \I\}$ and a set of simple coroots $\Pi^\vee=\{\alpha^\vee_i\,:\,i\in \I\}$. Let $\Z\I=\oplus_{i\in \I } \Z\alpha_i$ be the root lattice. Let $(\,,\,)$ be the normalized Killing form on $\Z\I$ so that the short roots have length $2$. The Weyl group $W$ is generated by the simple reflections $s_i:\Z\I\to \Z\I$, with $i\in \I$ defined by $s_i(\alpha_j)=\alpha_j-c_{ij}\alpha_i,$ for $j\in \I$. 
\\
\\
Let $q$ be an indeterminate and $\Q(q)$ be the field of rational functions in $q$ with coefficients in $\Q$. For each index $i\in\I$ and $n\in \Z_{\geq0}$, set
$$q_i:= q^{d_i},\qquad [n]_i:=\cfrac{q_i^n-q_i^{-n}}{q_i-q_i^{-1}}\qquad\text{and}\qquad [n]_{i}!:=\prod_{k=1}^n[k]_i.$$
Let $(Y,X, \langle \,,\, \rangle.\dots)$ be a root datum of type $(\I,\cdot)$; cf. \cite[\mbox{2.2}]{Lusztig2010}.  By definition, there exist embedding $\I\to X$, $i\mapsto \alpha_i$ and $\I\to Y$, $i\mapsto \alpha_i^{\vee}$. The relation between $\langle\,,\,\rangle$ and $(\,,\,)$ is $\langle i,j'\rangle=\frac{2(\alpha_i,\alpha_j)}{(\alpha_i,\alpha_i)}$, for $i,j\in \I$.
The quantum group $\uq$, associated to the root datum, is the unital associative algebra over $\Q(q)$, generated by the symbols $E_i,F_i$ for $i\in \I$ and $K_{h}$ for $h\in Y$, these symbols are subject to the following relations for $ i,j\in\I,$ and $ h,h_1,h_2\in Y:$
\begin{align*}
	K_0&=1,\\
	K_{h_1}K_{h_2}&=K_{h_1+h_2},\\
	K_{h}E_i&= q^{\langle h,\alpha_i\rangle}E_iK_h,\\
	K_{h}F_i&= q^{-\langle h,\alpha_i\rangle}F_iK_h,\\
	E_iF_j-F_jE_i&=\delta_{i,j}\cfrac{K_i-K_i^{-1}}{q_i-q_i^{-1}},
\end{align*}
as well as quantum Serre relations for $i\neq j$
\begin{align*}
	\sum_{r+s=1-a_{i,j}}(-1)^sE_i^{(r)}E_jE_i^{(s)}=\sum_{r+s=1-a_{i,j}}(-1)^s F_i^{(r)}F_jF_i^{(s)}=0.
\end{align*}
Here we use the notation
$$K_i:=K_{d_i\alpha_i^\vee},\qquad E_i^{(n)}:=\frac{1}{[n]_i!}E_i^n\qquad\text{and}\qquad F_i^{(n)}=\frac{1}{[n]_i!}F_i^n.$$
The quantum group $\uq$ has the structure of a Hopf-algebra $(\uq, \triangle,\epsilon,\iota,S,)$ via
\begin{align*}
	\triangle (K_i)&=K_i\otimes K_i,\\
	\triangle (E_i)&=E_i\otimes 1+K_i\otimes E_i,\\
	\triangle (F_i)&=1\otimes F_i+F_i\otimes K_i^{-1},\\
	\epsilon(E_i)&=\epsilon(F_i)=0,&& \epsilon(K_i)=1,\\
	\iota(f)&=f\cdot 1,\\
	S(E_i)&=-K_i^{-1}E_i,&&S(F_i)=-F_iK_i,&&S(K_h)=K_{-h},
\end{align*}
where $i\in\I$, $f\in \Q(q)$ and $h\in Y$.\\
\\
We refer to $\uq^+$, $\uq^-$ and $\uq^0$ as the $\Q(q)$ subalgebras generated by $E_i$, $F_i$ for $i\in \I$ and $K_{h}$ for $h\in Y$ respectively. We refer to involutive algebra anti-automorphisms as \textit{anti-involutions}. Let us record the (anti)-involutions 
\begin{enumerate}[(i)]
	\item There exists a $\Q(q)$-algebra anti-involution $\varrho:\uq\to\uq$ such that
	$$E_i\mapsto F_i,\qquad F_i\mapsto E_i,\qquad K_i\mapsto K_i,\qquad\qquad i\in\I.$$
	\item There exists a $\Q(q)$-algebra anti-involution $\sigma:\uq\to\uq$ such that
	$$E_i\mapsto E_i,\qquad F_i\mapsto F_i,\qquad K_i\mapsto K_i^{-1},\qquad\qquad i\in\I.$$
	\item There exists a $\Q(q)$-algebra involution $\omega:\uq\to\uq$ such that
$$E_i\mapsto F_i,\qquad F_i\mapsto E_i,\qquad K_i\mapsto K_i^{-1},\qquad\qquad i\in\I.$$
	\item There exists a $\Q$-algebra involution $\overline{\,\cdot\,}:\uq\to \uq$ such that
	$$E_i\mapsto E_i,\qquad F_i\mapsto F_i,\qquad K_i\mapsto K_i^{-1},\qquad q\mapsto q^{-1}.$$			
	We refer to the $\Q(q)$ involution mapping $q\mapsto q^{-1}$ as $\overline{\,\cdot\,}:\Q(q)\to \Q(q)$.
\end{enumerate}
We note that $\omega=\sigma\circ \varrho$.
\begin{rem}
	If we consider $\uq$ over the algebraic closure $\F$ of $\Q(q)$, all of these maps remain (anti)-involutions over $\F$.
\end{rem}
Checking that these maps are (anti)-involutions can be done by verifying that the relations are preserved, cf.
\cite[\mbox{4.6, 11.9}]{Jantzen1996}. We refer to the involution $\overline{\,\cdot\,}$ as the \textit{bar} involution. A $\Q$-linear operator on a $\Q(q)$ or $\F$ algebra is called \textit{anti-linear} if $q^m\mapsto q^{-m}$, for $m\in \Z$.
\subsection{Weight modules}\label{subsec:intromod}
The goal of this subsection is to recall the terminology and structure of finite dimensional $\uq$ weight modules and their bar involutions. 
A module $M$ is called a \textit{weight module} if $M$ has a decomposition of the form
$$M=\bigoplus _{\lambda\in X}M_{\lambda },\qquad\text{where}\qquad M_\lambda=\{m\in M \,:\, K_h m= q^{\langle h,\lambda \rangle}m, \, h\in Y\}.$$
We say that a vector $v\in M$ is a \textit{weight vector of weight $\lambda$} if $K_h v= q^{\langle h,\lambda \rangle}v$ for each $ h\in Y.$
A module $M$ is called a \textit{highest weight module} if there exists a weight vector $v\in M$ with the property that $\uq v=M$ and $E_iv=0$ for all $i\in \I.$ 
The irreducible finite dimensional highest weight left modules are parameterized by the dominant weights $X^+$, cf. \cite{Lusztig2010}. We write $L(\lambda)$ to denote the simple left module corresponding to a dominant weight $\lambda\in X^+$.
The theory of highest weight modules has a natural analog for right modules, cf. \cite[\mbox{3.5.7}]{Lusztig2010}.
\\
We denote by $\comp$ an algebraic completion of $\uq$, see \cite[\mbox{Sect 3}]{Balagovic2016} for the construction of $\comp$.
Let $M=L(\lambda)$ be an irreducible $\uq$ weight module and $v_\lambda$ be a choice of highest weight vector, then there exists a \textit{bar involution} $\overline{\,\cdot\,}^M: M\to M$ relative to $v_\lambda$  with the property that
$$\overline{X\cdot v_\lambda}^M:=\overline{X}\cdot v_\lambda\qquad\text{for}\qquad X\in \uq.$$
If it is clear from the context, we denote $\overline{\,\cdot\,}^M=\overline{\,\cdot\,}$. Analogously, for irreducible right $\uq$ weight modules $N$ there exist bar involutions relative to a choice of highest weight vector.
\subsection{Lusztig braid group operators}
In this section, we recall the Lusztig braid group operators, cf. \cite[\mbox{Ch. 5 \& 37}]{Lusztig2010}. Let $e=\pm 1$, and let $M$ be a $\uq$ weight module. The \textit{Lusztig braid group operators} are defined as $T'_{i,e}, T''_{i,e}: M_\mu \to M_\mu$, where
\begin{align*}
	T'_{i,e}(v) &= \sum_{a,b,c; \, a - b + c = \langle \alpha_i^\vee, \mu \rangle} (-1)^b q_i^{e(-ac + b)} F_i^{(a)} E_i^{(b)} F_i^{(c)} v \\
	T''_{i,e}(v) &= \sum_{a,b,c; \, a - b + c = \langle \alpha_i^\vee, \mu \rangle} (-1)^b q_i^{e(-ac + b)} E_i^{(a)} F_i^{(b)} E_i^{(c)} v,
\end{align*}
and the operators $T'_{i,e}$ and $T''_{i,e}$ satisfy the braid relations
$$
\underbrace{T'_{i,e} T'_{j,e} T'_{i,e} \dots}_{m_{ij} \text{ factors}} = \underbrace{T'_{j,e} T'_{i,e} T'_{j,e} \dots}_{m_{ij} \text{ factors}} \quad \text{and} \quad \underbrace{T''_{i,e} T''_{j,e} T''_{i,e} \dots}_{m_{ij} \text{ factors}} = \underbrace{T''_{j,e} T''_{i,e} T''_{j,e} \dots}_{m_{ij} \text{ factors}},
$$
where $m_{ij}$ denotes the order of $s_i s_j \in W$. If $w = s_{i_1} \dots s_{i_n} \in W$ is a reduced expression, there exist well-defined operators $T'_{w,e}$ and $T''_{w,e}$ given by
$$
T'_{w,e} = T'_{i_1,e} \cdots T'_{i_n,e} \quad \text{and} \quad T''_{w,e} = T''_{i_1,e} \cdots T''_{i_n,e}.
$$
We also use the symbols $T'_{i,e}$ and $T''_{i,e}$ to denote algebra automorphisms on $\uq$, cf. \cite[\mbox{Ch. 37}]{Lusztig2010}. The Lusztig braid group operators on the modules and on $\uq$ are related by
\begin{equation}\label{eq:lusztigaut}
T'_{i,e}(xv) = T'_{i,e}(x)T'_{i,e}(v) \, \text{and} \, T''_{i,e}(xv) = T''_{i,e}(x)T''_{i,e}(v) \quad \text{for} \quad v \in M, x \in \uq.
\end{equation}
Hence, the algebra automorphisms $T'_{i,e}$ and $T''_{i,e}$ on $\uq$ are given by conjugation by the invertible elements $T'_{i,e} \in \comp$ and $T''_{i,e} \in \comp$. Therefore, the Lusztig braid group operators on $\uq$ also satisfy the braid relations. By \cite[\mbox{Ch. 37.2}]{Lusztig2010}, the following relations hold.
\begin{align}\label{eq:lusztigrel}
	T''_{i,e}=(T'_{i,-e})^{-1}&&T''_{i,e}\circ \sigma= \sigma\circ T'_{i,-e}\\
	T''_{i,e}\circ \overline{\,\cdot\,}= \overline{\,\cdot\,}\circ T''_{i,-e}&&T'_{i,e}\circ \overline{\,\cdot\,}= \overline{\,\cdot\,}\circ\nonumber T'_{i,-e}\\
	T''_{i,e}\circ \omega= \omega \circ T'_{i,e}.\nonumber
\end{align}
\subsection{The quantized function algebra and bilinear forms}\label{subsec:qfa}
In this subsection, we recall the definition and structure of the quantized function algebra.
The \textit{quantized function algebra} $\qfa$ is the subspace of $\uq^{\ast}$ spanned by matrix coefficients for finite dimensional $\uq$ weight modules.
The quantized function algebra has the structure of a
Hopf-algebra $(\qfa, \triangle,\epsilon,\iota,S)$ which yields a Hopf-algebra duality between $\uq$ and $\qfa$, \cite[\mbox{Ch 7}]{Jantzen1996}. We continue by describing the structure of $\qfa$ as a $\uq$ bi-module. For each dominant weight $\lambda\in X^+$, the dual $L(\lambda)^\ast$ is a right $\uq$ weight module by setting
$$(fX)(v):=f(Xv)\qquad \text{where}\qquad X\in \uq,\, f\in L(\lambda)^\ast,\, v\in L(\lambda).$$
We define $c^{L(\lambda)}:L(\lambda)^\ast\otimes L(\lambda)\to \qfa$ by
$$c^{L(\lambda)}_{f,v}(X):=f(X\cdot v).\qquad\text{where}\qquad f\in L(\lambda)^\ast,v\in L(\lambda), \, X\in \uq.$$
If the module is clear from the context, we write $c_{f,v}$. With these conventions, the quantized function algebra has the quantum Peter-Weyl decomposition
$$\qfa=\bigoplus_{\lambda\in X^+} L(\lambda)^\ast\otimes L(\lambda)$$
as a $\uq$-bi-module, cf. \cite[\mbox{7.2}]{Kashiwara1993}. The dual $L(\lambda)^\ast$ has a bar involution defined by
\begin{equation}\label{eq:barinv}
	\overline{f}(v):= \overline{f(\overline{v})},\qquad \text{where}\qquad f\in L(\lambda)^\ast , v\in L(\lambda).
\end{equation}
Next we describe an explicit identification of $L(\lambda)^\ast$ via the Shapovalov form. For this, we fix notation regarding twisting modules with automorphisms.
\begin{defi}[Twisting modules with automorphisms]
	Let $\xi$ be an automorphism of $\uq$ and $\eta$ be an anti-automorphism of $\uq$. Let $M$ be a left module of $\uq$ and $N$ be a right module of $\uq$
	\begin{enumerate}[(i)]
		\item 	We define the left module $^\xi M$ by $^\xi M\cong M$ as vector spaces and 
		$X\cdot m:= \xi(X)\cdot m$ for $ X\in\uq, m\in{^\xi} M.$ 
		\item We define the right module $^\eta M$ by $^\eta M\cong M$ as vector spaces and
		$m\cdot  X:= \eta(X)\cdot m$ for $ X\in\uq, m\in {^\eta}M.$ 
		\item 	We define the right module $^\xi N$ by $^\xi N\cong N$ as vector spaces and
		$n\cdot X:=n \cdot \xi(X)$ for $ X\in\uq, n\in{^\xi}  N.$ 
		\item We define the left module $^\eta N$ by $^\eta N\cong N$ as vector spaces and
		$X\cdot n:= n\cdot\eta(X) $ for $ X\in\uq, n\in {^\eta}N.$ \qedhere
	\end{enumerate}
\end{defi}
We define $L^r(\lambda)$ to be the simple right module $^\varrho L(\lambda)$.
In \cite[\mbox{7.1.3}]{Kashiwara1993} Masaki Kashiwara notes that for each positive weight $\lambda\in X^+$, the irreducible weight module $L(\lambda)$ has a unique non-degenerate symmetric bi-linear form, relative to a choice of highest weight vector $v_\lambda \in L(\lambda)$, with $(v_{\lambda},v_{\lambda})=1$ and 
\begin{equation}\label{eq:Shapovalof}
	(v,X\cdot w)=(\varrho(X)\cdot v,w)\qquad \text{where}\qquad v,w\in L(\lambda),\, X\in\uq.
\end{equation}
We refer to this bilinear form as the \textit{Shapovalov form}. Via this identification we have $L(\lambda)^\ast\cong L^r(\lambda)$ as right $\uq$ modules.
There are two natural \textit{bar involutions} on the quantized function algebra.
We define the \textit{bar involution} $\psi:\qfa\to \qfa$
by 
$$\psi(\varphi)(X):= \overline{\varphi(\overline{X})},\qquad\text{where}\qquad X\in \uq,\varphi\in\qfa.$$
We define a second \textit{bar involution} $\bar{\,\,\cdot\,\,}^{\qfa}:\qfa\to \qfa$ as follows: if $\varphi=c_{f,v}$ with $f\in L^r(\lambda)$ and $v\in L(\lambda)$, then set $\overline{c_{f,v}}^{\qfa}:=c_{\overline{f}, \overline{v}}\in R_q(G)$ and extend anti-linearly to $\qfa$.
In \cite[\mbox{7.3.4}]{Kashiwara1993}, it is noted that the following diagram commutes
\begin{equation}\label{prop:kashi}
	\begin{tikzcd}
		{L(\lambda)^r\otimes_{\Q(q)}L(\lambda)} \arrow[r, "c^{L(\lambda)}"] \arrow[d, "\textendash\,\otimes\,\textendash"'] & R_q(G) \arrow[d, "\psi"] \\
		{L(\lambda)^r\otimes_{\Q(q)}L(\lambda)} \arrow[r, "c^{L(\lambda)}"]                                   & R_q(G)                    
	\end{tikzcd}.
\end{equation}
In particular, with respect to the Shapovalov form we have
$$(\overline{f},\overline{v})=c_{\overline{f},\overline{v}}(1)\stackrel{{( \ref{prop:kashi})}}{=}\psi(c_{f,v}(1))=\overline{c_{f,v}(\overline{1})}=\overline{(f,v)},\qquad\text{for}\qquad f\in L^r(\lambda),v\in L(\lambda).$$ 
If it is clear from the context, we leave out the superscript of the bar involution. 
\section{Quantum symmetric pairs and $\iota$-bar involutions}\label{sec:iota}
\subsection{Quantum symmetric pairs}\label{subsec:qsp}
In this section, we introduce quantum symmetric pairs in the sense of \cite{Letzter1999}. 
This section mostly follows \cite{Kolb2014}.
\\
\\
By $\text{Aut}(C)$, we denote the set of permutations $\tau:\I\to \I$ that satisfy $c_{ij}=c_{\tau(i)\tau(j)}$ for all $i,j\in \I$. The set of permutations that satisfy $\tau(\I_\bullet)=\I_\bullet$ is denoted  $\text{Aut}(C,\I_\bullet)$.
Let $\I_\bullet\subset \I$ be of finite type and $\tau\in\text{Aut}(C,\I_\bullet)$. Then the pair $(\I_\bullet,\tau)$ is called an \textit{admissible pair} if the following three conditions are satisfied:
\begin{enumerate}[(i)]
	\item $\tau^2=\text{id}_I$;
	\item the action of $\tau$ on $\I_\bullet$ coincides with $-w_\bullet;$
	\item if $i\in \I_\circ:=\I\setminus \I_\bullet$ satisfies $\tau(i)=i$ then $\langle \rho_\bullet^\vee,\alpha_j\rangle\in \Z,$ where $\rho^\vee_\bullet\in Y$ is the half sum of the positive coroots relative to $\mathfrak{g}_\bullet$.
\end{enumerate}
Given an admissible pair $(\I_\bullet,\tau)$, we denote $\Theta:=-w_\bullet \circ \tau$, which naturally acts on $X$, $Y$ and $\Z \I$. By $\uq_\bullet$, we denote the $\uq$ subalgebra generated by $E_i,F_i$ and $K_i$ with $i\in \I_\bullet$.  
Let us define	
$$\I_{\mathrm{ns}}:=\{i\in \I_\circ\,:\, \tau(i)=i\text{ and } c_{ij}=0\text{ for all } j\in \I_\bullet\}.$$
With this in place, we define parameters, which, as stated, are part of the datum for quantum symmetric pairs. A \textit{parameter} is a tuple $(\boldsymbol{c},\boldsymbol{s})$ where 
\begin{align*}
	\boldsymbol{c}\in \mathcal C &:=\{\boldsymbol{c}\in (\Q(q)^\times)^{\I_\circ}\,:\, c_i=c_{\tau(i)}\text{ if }\tau(i)\neq i\text{ and } (\alpha_i,\Theta(\alpha_i))=0\}\quad\text{and}\\
	\boldsymbol{s}\in \mathcal S &:=\{\boldsymbol{s}\in (\Q(q)^\times)^{\I_\circ}\,:\, s_j\neq 0\implies (j\in \I_{\mathrm{ns}}\text{ and  }c_{ij}\in -2\N_0\,\forall i\in \I_{\mathrm{ns}}\setminus \{j\})\}.
\end{align*}
\begin{enumerate}[(i)]
	\item We call a parameter $(\boldsymbol{c},\boldsymbol{s})$  \textit{balanced} if $\boldsymbol{s}=0$ and  $c_i=c_{\tau(i)}$ for each $i\in\I_\circ$.
	\item We call a parameter $(\boldsymbol{c},\boldsymbol{s})$ \textit{uniform} if
	\begin{align}\label{eq:uniform}
		&c_{i}=(-1)^{\langle 2\rho^\vee_\bullet,\alpha_i\rangle}q^{-(\alpha_i,w_\bullet \alpha_{\tau(i)}+2\rho_{\bullet})} \overline{c_{\tau(i)}}&&\text{for each }i\in \I_{\circ}.
	\end{align}
	\item We call a parameter $(\boldsymbol{c},\boldsymbol{s})$ \textit{admissible} if $\boldsymbol{c}$ is balanced, $c_i^{-1}=\overline{c_i}$ and
	\begin{align}\label{eq:admissible}
		&c_{i}=q^{-(\alpha_i,w_\bullet \alpha_{\tau(i)}+2\rho_{\bullet})} \overline{c_{\tau(i)}}&&\text{for each }i\in \I_{\circ}.
	\end{align}
\end{enumerate}
The \textit{distinguished parameter} $\boldsymbol{c}^{\diamondsuit}$ is the unique balanced parameter defined by
\begin{align}\label{eq:distinguished}
	c^\diamondsuit_i=-q^{-(\alpha_i,\alpha_i-\Theta(\alpha_i))/2}\qquad\qquad \text{for }i\in \I_\circ.
\end{align}
This parameter is central in the work \cite{Wang2023}. 
Let us denote by $\uq_{\Theta}^0$ the commutative algebra generated by $K_iK_{\tau(i)}^{-1}$ with $i\in \I_\circ$ and $K_j$ with $j\in\I_\bullet$.
\begin{defi}[Quantum symmetric pair]
	Let $(\I_\bullet,\tau)$ be an admissible pair and $(\boldsymbol{c},\boldsymbol{s})$ be a parameter, then we define the \textit{quantum symmetric pair coideal subalgebra} $\uqbs$ as the subalgebra of $\uq$ generated by $\uq_\bullet$, $\uq_{\Theta}^0$ and the elements\\
	$
	B_i=F_i+c_i T''_{w_\bullet,+1}(E_{\tau(i)})K_i^{-1}+s_iK_i^{-1}$ for each $  i\in \I_\circ$.\qedhere
\end{defi}
The algebra $\uqbs$ is a right coideal subalgebra, meaning that
$$\triangle (\uqbs)\subset \uqbs\otimes \uq.$$
We say that $\uqbs$ is \textit{standard} if $s_i=0$ for each $i\in\I_\circ$. If $\boldsymbol{s}=0$, we denote $\uqbs=\uqb$. Usually, we consider the quantum group and the coideal as algebras over $\Q(q)$. Quantum groups can be defined over other fields, such as the algebraic closure $\F$ of $\Q(q)$, cf. \cite{Wang2023}. In the setting of quantum groups over $\F$ one observes the following relation between distinct quantum symmetric pair coideal subalgebras $\uqb$ and $\uqd$, cf. \cite[\mbox{Sec 3,2}]{Kolb2014}.
Define the Hopf-algebra automorphism $\Phi_{\boldsymbol{a}} :\uq\to\uq$ on the $\F$-algebra $\uq$ by
\begin{align}\label{eq:iso}
	K_i\mapsto K_i\qquad E_i\mapsto a_i^{1/2}E_i\qquad\text{and}\qquad F_i\mapsto  a_i^{-1/2}F_i
\end{align}
for a scalar tuple $\boldsymbol{a}\in (\F^{\times})^{\I}$ with $a_i=1$ for $i\in\I_\bullet$.  
Then $\Phi_{\boldsymbol{a}}$ defines a $\F$-algebra isomorphism
$\Phi_{\boldsymbol{a}}:\uqb\to\uqd$ where $d_i=c_i\cdot a_i^{1/2}a_{\tau(i)}^{1/2}$ for $i\in\I$.\\
Let $(\I,\I_\bullet,\tau)$ be a Satake diagram. For each $i\in \I_\circ$, the datum $(\{i,\tau(i)\}\cup \I_\bullet,\I_\bullet,\tau)$ yields a Satake diagram, cf. \cite[\mbox{Rem 2.2}]{Dobson2019}. For each $i\in\I_\circ$, we denote by $\uq_i$ the quantum group generated by $E_j,F_j$ and $K_j$ for $j\in \{i,\tau(i)\}\cup \I_\bullet$. If $(\boldsymbol{c},\boldsymbol{s})$ is a parameter, then we denote by $\uqbs^i$ the coideal subalgebra for the parameter $(\boldsymbol{c}|_{\{i,\tau(i)\}},\boldsymbol{s}|_{\{i,\tau(i)\}})$ and Satake diagram $(\{i,\tau(i)\}\cup \I_\bullet,\I_\bullet,\tau)$. We view $\uq_i\subset \uq$ and $\uqbs^i\subset \uqbs$ as subalgebras. 
We define 
\begin{equation}\label{eq:thethainv}
	Y_{\Theta}=\{h\in Y:\Theta(h)=-h\}.
\end{equation} 
The \textit{quantum torus} $\mathcal A\subset \uq^0$ is the subalgebra generated by the $K_h$ with $h\in Y_\Theta$. Through restricting matrix entries to the quantum torus, we obtain a map $\Res: \qfa\to \Hom(\mathcal A,\F)$. We use $\qfaa$ to denote the image of $\Res$.
Next, we introduce an action of $\frac{1}{2}Y$ on $\qfa$, and on parameters $(\boldsymbol{c},\boldsymbol{s})$. We need this action to make sense of $\rho=\frac{1}{2}\sum_{\alpha>0}\alpha$ acting on parameters. Let $\widehat{\uq^0}$ be the group algebra generated by the lattice $\frac{1}{2}Y$. Let $K_h$ with $h\in\frac{1}{2}Y$ denote the basis elements of $\widehat{\uq^0}$.
We extend the Killing form $Y$ to $\frac{1}{2}Y$ in the canonical way. Using this extension, each $\uq$ left weight module $M$ becomes a $\widehat{\uq^0}$ module by defining
$$K_h v_{\mu}:=q^{\langle h,\mu\rangle}v_\mu \qquad\text{where}\qquad v_\mu\in M_\mu\text{ and } h\in \frac{1}{2}Y.$$
Similarly, if $M$ is a right module, we write the action from the right. For the $\uq$ weight module $\uq$, we write $\mathrm{ad}(K_h)$ to denote the action on $\uq$. 
Therefore, $\widehat{\uq^0}$ acts on $\qfa$ through
$$ c^{L(\lambda)}_{f,v}\triangleleft K_h(X):=c^{L(\lambda)}_{fK_h,v}(X),\qquad\qquad K_h\in \widehat{\uq^0},c^{L(\lambda)}_{f,v}\in\qfa,\text{ and }X\in\uq.$$
Lastly we introduce an action of $\frac{1}{2}Y$ on parameters. Let $h\in \frac{1}{2}Y$ and let $(\boldsymbol{c},\boldsymbol{s})$ be a parameter. Then, define the parameter $\mathrm{ad}(K_h)(\boldsymbol{c},\boldsymbol{s})$ by
$$ \mathrm{ad}(K_h)(c)_i=c_i q^{\langle h,\alpha_i-\Theta(\alpha_i)\rangle}\quad \text{and}\quad \mathrm{ad}(K_h)(s)_i =s_iq^{\langle h,\alpha_i\rangle}\qquad\text{for}\qquad i\in \I_\circ.$$
Using this notation, we have $\mathrm{ad}(K_h)(\uqbs)=\mathbf{B}_{\mathrm{ad}(K_h)(\boldsymbol{c},\boldsymbol{s})}$. We continue by introducing $\imath$-bar involutions and quasi-$K$ matrices.
\begin{thm}[$\imath$-Bar involution]\label{thm:barinvolution}
	\cite[\mbox{Thm 3.11}]{Balagovic2015}
	Let $(\uqb,\uq)$ be a quantum symmetric pair and $\boldsymbol{c}$ be uniform parameter, then there exists an anti-linear involution $\overline{\,\cdot\,}^{\uqb}:\uqb\to\uqb$, which we will refer to as the $\imath$-\textit{bar-involution}, that satisfies
	\begin{align*}
		K_i\mapsto K_i^{-1},&&E_i\mapsto E_i,&&F_i\mapsto F_i,&& \text{for }i\in \I_{\bullet},\\
		&&K_iK_{\tau(i)}^{-1}\mapsto K_i^{-1}K_{\tau(i)},&&\text{ and } B_i\mapsto B_i&& \text{for }i\in \I_{\circ}.
	\end{align*}
\end{thm}
\begin{thm}[Intertwiner]\label{thm:intertwiner}
	\cite[\mbox{Thm 6.10.}]{Balagovic2016}
	Suppose that $\boldsymbol{c}$ is a uniform parameter, then there exists a unique element $\Upsilon=\sum _{\mu\in \N\I}\Upsilon _\mu \in \comp$ that satisfies $\Upsilon_\mu \in \uq^+_{\mu}$, $\Upsilon_0=1$ and  $\overline{x}^{\uqb}\Upsilon=\Upsilon \overline{x}$ for each $x\in \uqb$.
\end{thm}
We refer to the intertwiner as the \textit{quasi-$K$} matrix.
In particular it holds that $x\Upsilon=\Upsilon x$ for $x\in \uq_\bullet\uq^0_\Theta$.
For the intertwiner of Theorem \ref{thm:intertwiner} it holds that $\overline{\Upsilon}=\Upsilon^{-1}$, cf. \cite[\mbox{Prop 3.5}]{Wang2023}. 
\begin{prop}[$\imath$-bar involutions]\cite[Prop 5.1]{Bao2018}
	Let $\lambda\in X^+$ be a dominant weight and $\boldsymbol{c}$ be a uniform parameter. 
	Then $v\mapsto\Upsilon \overline{v}$ is an anti-linear involution of $L(\lambda)$ and $f\mapsto \overline{f}\Upsilon^{-1}$, $f\in L(\lambda)^\ast$, is an anti-linear involution of $L(\lambda)^\ast.$
\end{prop}
We refer to these involutions as the $\imath$-bar involutions.
\subsection{$\iota$-Braid group operators and the relative Weyl group}\label{subsec:invo}
The goal of this subsection is to introduce the braid group operators of the quantum symmetric pair coideal subalgebra $\uqb$ as well as their action on irreducible $\uq$ weight modules. This subsection is based on \cite{Wang2023}.
\begin{rem}
	Unless stated otherwise, we assume the quantum symmetric pair to be standard and over $\F$.
\end{rem}
Let $\Sigma$ be the \textit{restricted root system} of $(\mathfrak{g},\Theta)$ and denote by $W^\Sigma$ the relative Weyl group, see \cite[\mbox{Sec 2.3}]{Dobson2019} for a comprehensive survey. Let $J\subset \I$ be a subset, then we denote by $w_J$ the longest element of the parabolic subgroup $W_J$. 
\begin{thm}\cite[\mbox{Prop 2.7}]{Dobson2019}\label{thm:dobsonkolb}
	The \textit{relative Weyl group} $W^\Sigma$ can be embedded in $W$, and the image is generated by the elements $\boldsymbol{r_i}:=w_{\{i,\tau(i)\}\cup \I_\bullet}w_\bullet^{-1}$ 
	where $i\in \I_\circ$.
\end{thm}
By a slight abuse of notation we will almost always view $W^\Sigma$ as a subgroup of the Weyl group $W$. The relative Weyl group acts on $\Hom(\mathcal A,\F)$ by 
$$[w\cdot\varphi](K_h):= \varphi(K_{w\cdot h}),\qquad \qquad \text{for}\qquad w\in W^\Sigma,\, \varphi\in\Hom(\mathcal A,\F), h\in Y_\Theta$$
and extend linearly to $\mathcal A$.
Let $\Hom(\mathcal A,\F)^{W^\Sigma}$ denote the set of invariants. We call their elements \textit{Weyl group invariant}. 
We say that a matrix entry $\varphi\in \qfa$ is \textit{invariant under the affine Weyl group action} if $\varphi \triangleleft K_{\rho}\in\qfaaws$. This invariance is understood as
$$\varphi(K_h)=\varphi(K_{w(h-\rho)+\rho})\qquad \text{for}\qquad h\in Y_\Theta,w\in W^\Sigma.$$
\begin{defi}[Rescaled Lusztig braid group operators]
	Define for $i\in \I_\circ$ the \textit{rescaled Lusztig braid group operator} 
	\begin{align}
		\mathcal T'_{i,-1;\boldsymbol{c}}:=\Phi_{\overline{\boldsymbol{c}^{\diamondsuit}}\boldsymbol{c}}T_{i,-1}' \Phi_{\overline{\boldsymbol{c}^{\diamondsuit}}\boldsymbol{c}}^{-1}.
	\end{align}
	Here, $\Phi_{\overline{\boldsymbol{c}^{\diamondsuit}}\boldsymbol{c}}^{-1}$ is the operator from (\ref{eq:iso}), and we we understand $\overline{\boldsymbol{c}^{\diamondsuit}}\boldsymbol{c}$ as component-wise multiplication.
\end{defi}
Let $M$ be a finite dimensional $\uq$ weight module. Recall from (\ref{eq:lusztigaut}) that 
\begin{equation}\label{eq:braidlus}
T'_{\boldsymbol{r_i},-1}(x)v= T'_{\boldsymbol{r_i},-1}x (T'_{\boldsymbol{r_i},-1})^{-1}v,\qquad \qquad\text{for}\qquad x\in \uq, v\in M,\,i\in\I_\circ,
\end{equation} 
here the operator $T'_{\boldsymbol{r_i},-1}\in \comp$ acts locally finite on $M$ according to an element $T'_{\boldsymbol{r_i},-1}\in \uq$, cf. \cite[\mbox{Prop 37.1.2}]{Lusztig2010}. Using that $T'_{\boldsymbol{r_i},-1}\in \uq$, allows us to define the elements $\Phi_{\overline{\boldsymbol{c}^{\diamondsuit}}\boldsymbol{c}}(T'_{\boldsymbol{r_i},-1})\in \uq$ and $\Phi_{\overline{\boldsymbol{c}^{\diamondsuit}}\boldsymbol{c}}\big((T'_{\boldsymbol{r_i},-1})^{-1}\in \uq$. Thus we see that
$$\mathcal T'_{\boldsymbol{r_i},-1;\boldsymbol{c}}(x)v=\Phi_{\overline{\boldsymbol{c}^{\diamondsuit}}\boldsymbol{c}}(T'_{\boldsymbol{r_i},-1})x\Phi_{\overline{\boldsymbol{c}^{\diamondsuit}}\boldsymbol{c}}\big((T'_{\boldsymbol{r_i},-1})^{-1}\big)v\qquad\text{for}\qquad x\in \uq,v\in M,\, i\in\I_\circ.$$
From the above it follows that the operator $\mathcal T'_{\boldsymbol{r_i},-1;\boldsymbol{c}}$ extends to $\comp$ via
$$\mathcal T'_{\boldsymbol{r_i},-1;\boldsymbol{c}}:\comp\to \comp,\qquad \qquad x\mapsto \Phi_{\overline{\boldsymbol{c}^{\diamondsuit}}\boldsymbol{c}}(T'_{\boldsymbol{r_i},-1})x\Phi_{\overline{\boldsymbol{c}^{\diamondsuit}}\boldsymbol{c}}\big((T'_{\boldsymbol{r_i},-1})^{-1}\big)\quad\text{for}\quad x\in \comp,\, i\in\I_\circ.$$
Analogously to (\ref{eq:braidlus}) for $i\in \I_\circ$ we use $\mathcal T'_{\boldsymbol{r_i},-1;\boldsymbol{c}}$ to denote the operator $\Phi_{\overline{\boldsymbol{c}^{\diamondsuit}}\boldsymbol{c}}(T'_{\boldsymbol{r_i},-1})\in \End(M)$.
Let $i\in \I_\circ$. Then we set $\Upsilon_{i,\boldsymbol{c}}$ to be the the quasi-$K$ matrix associated to the rank one Satake-subdiagram induced by $i$, cf. \cite[\mbox{Sec 3.3}]{Wang2023}. 
\begin{thm}[Braid group operators]\cite[\mbox{Thm 9.10, Prop 10.1}]{Wang2023}\label{thm:braid}
	Let $\boldsymbol{c}$ be a balanced parameter. Then there exist a homomorphism
	\begin{align}
		\boldsymbol{T}'_{-,-1}:\Br(W^\Sigma)\to \mathrm{Aut}_{\mathrm{alg}}(\comp),
	\end{align}
	with the following two properties:
	\begin{enumerate}[(i)]
		\item When restricted to $\uqb$, it holds that $\boldsymbol{T}'_{-,-1}:\Br(W^\Sigma)\to \mathrm{Aut}_{\mathrm{alg}}(\uqb)$ is a homomorphism.
		\item For all $x\in \comp$ and $i\in\I_\circ$, it holds that	$\boldsymbol{T}_{i,-1}' (x)\Upsilon_{i,\boldsymbol{c}}=\Upsilon_{i,\boldsymbol{c}} \mathcal T'_{\boldsymbol{r_i},-1;\boldsymbol{c}}(x).$
	\end{enumerate}
\end{thm}
\begin{proof}
	The first statement follows by \cite[\mbox{Thm 9.10, Prop 10.1\& Thm 10.6}]{Wang2023}. The intertwining properties for $x\in \uqb$ and $i\in \I_\circ$ are given by
	$
	\boldsymbol{T}_{i,-1}' (x)\Upsilon_{i,\boldsymbol{c}}=\Upsilon_{i,\boldsymbol{c}} \mathcal T'_{\boldsymbol{r_i},-1;\boldsymbol{c}}(x).
	$
	We take this factorization as a definition, i.e., we define the operator
	$$
	\boldsymbol{T}_{i,-1}' (x):=\Upsilon_{i,\boldsymbol{c}} \mathcal T'_{\boldsymbol{r_i},-1;\boldsymbol{c}}(x)\Upsilon_{i,\boldsymbol{c}}^{-1}\in\comp\qquad\text{for}\qquad x\in \comp,i\in\I_\circ.\qedhere
	$$
\end{proof}
\begin{thm}[Braid group operators on modules]
	Let $M$ be a finite dimensional $\uq$ weight module and $\boldsymbol{c}$ be a balanced parameter. Then there exist representations 
	\begin{align}
		\boldsymbol{T}'_{-,-1}:\Br(W^\Sigma)\to \mathrm{Aut}_\F(M)
	\end{align}
	satisfying
	\begin{align}\label{eq:braidrel}
		\boldsymbol{T}'_{i,-1}(xm)&=\boldsymbol{T}'_{i,-1}(x)\boldsymbol{T}'_{i,-1}(m)\qquad\text{for}\qquad x\in \uq,m\in M,i\in\I_\circ.
	\end{align}
\end{thm}
\begin{proof}
	The existence of the representation is proved in \cite[\mbox{Thm 10.5 \& Thm 10.6}]{Wang2023}. To be precise, the operator is defined by
	\begin{align*}
		\boldsymbol{T}'_{i,-1}(m)&=\Upsilon_{i,\boldsymbol{c}}\mathcal T'_{\boldsymbol{r_i},-1;\boldsymbol{c}}(m)\qquad\text{for}\qquad m\in M,i\in\I_\circ.
	\end{align*}
	The relation (\ref{eq:braidrel}) follows from the definition of $\boldsymbol{T}'_{i,-1}$ in Theorem \ref{thm:braid}.
\end{proof}
\begin{rem}
	To denote the parameter dependence, the Wang-Zhang braid group operators  and quasi-$K$ matrix have a subscript $\boldsymbol{c}$. However, if it is clear from the context, we leave out the subscript.
\end{rem}
\begin{defi}
	Let $\boldsymbol{c}$ be a balanced parameter. For each $\boldsymbol{r_i}\in W^\Sigma$, we define the precomposition operator 
	$-\circ\boldsymbol{T}'_{i,-1}:\qfa\to\qfa$ as follows: if $c^{L(\lambda)}_{f,v}\in L(\lambda)^\ast\otimes L(\lambda)$, then for $X\in \uq$ we set
	\begin{align*}
		c^{L(\lambda)}_{f,v} \circ \boldsymbol{T}'_{i,-1}(X)&:=c_{f \Upsilon_{i,\boldsymbol{c}},\Upsilon_{i,\boldsymbol{c}}^{-1} v}(\mathcal T'_{\boldsymbol{r_i},-1;\boldsymbol{c}}(X)),
	\end{align*}
	and we extend this linearly to $\qfa$. 
	\qedhere
\end{defi}
\section{Characters of $\boldsymbol{\uqbs}$ and spherical functions}\label{sec:repth}
\subsection{Characters of $\boldsymbol{\uqbs}$}\label{sec:char}
In this section we study the structure of characters that occur in restrictions of finite dimensional $\uq$ weight modules. Our goal is to understand their embedding and their relation relation with the $\iota$-bar involution.  
\\
\\
Let $\lambda\in X^+$. Recall that $v_\lambda$ is a nonzero vector in $L(\lambda)_\lambda$ and fix a nonzero bar invariant vector $v_{w_0\lambda}$ in $L(\lambda)_{w_0\lambda}$. Remark that both vectors are unique up to a nonzero scalar multiple. Similarly we denote by $v_\lambda^\ast$ a nonzero vector in $L(\lambda)^\ast$ dual to $v_\lambda$ and by $v_{w_0\lambda}^\ast\in L(\lambda)^\ast$ a nonzero vector dual to $v_{w_0\lambda}$. The following proposition is based on \cite[\mbox{Lem 3.3}]{Letzter2003} and \cite[\mbox{Prop 8.1}]{Kolb2014}. 
\begin{prop}\label{prop:weight}
	\begin{enumerate}[(i)]
		\item Let $\lambda\in X^+$ be a dominant weight and suppose that $v\in L(\lambda)$ spans a one-dimensional $\uqbs$ submodule. Then, up to a nonzero scalar multiple, $v$ equals $v_{\lambda}+w+ a v_{w_0\lambda}$ for a nonzero scalar $a\in \F$ and $w\in \bigoplus_{w_0\lambda<\mu<\lambda}L(\lambda)_\mu$.
		\item Let $\lambda\in X^+$ be a dominant weight and suppose that $f\in L(\lambda)^\ast$ spans a one-dimensional $\uqbs$ submodule. Then, up to a nonzero scalar multiple, $f$ equals \\$v_{\lambda}^\ast+w+ a v^\ast_{w_0\lambda}$ for a nonzero scalar $a\in \F$ and $w\in \bigoplus_{w_0\lambda<\mu<\lambda}L(\lambda)^\ast_\mu$.
	\end{enumerate}
\end{prop}
\begin{proof}
	Let us write $v=\sum_\mu v_\mu$ for weight vectors $v_\mu\in L(\lambda)_\mu$. We define the set $\text{Supp}(v):=\{\mu\in X\,:\, v_\mu\neq 0\}$. Let $\eta$ be a maximal element of $\text{Supp}(v)$, i.e.  $\eta+\alpha_i\notin\text{Supp}(v)$ for simple roots $\alpha_i$. Denote by $\chi:\uqbs\to \F$ the representation corresponding to the one-dimensional $\uqbs $ module spanned by $v$.
	For $i\in \I_\bullet$, one has $E_i\in \uqbs$ and therefore
	\begin{align*}
		\chi(E_i)v_\eta+l.o.t.&=\chi(E_i)v\\
		&=E_iv\\
		&=E_iv_{\eta}+l.o.t.\,.
	\end{align*}
	Because the weight of $E_i v_\eta$ is greater than the weight of $v_\eta$, and because of the maximality of $\eta$, this leads to $E_iv_{\eta}=0$. Since the only one-dimensional weight module of $\uq_\bullet$ is $\epsilon$, we have $F_jv=0=E_jv$ for each $j\in \I_\bullet$. As a result, one has $(T''_{w_\bullet,+1})^{-1}v=v$. Using this fact, we compute
	for $i\in \I_\circ$ that
	\begin{align*}
		\chi(B_i)v_\eta+l.o.t.&=\chi(B_i)v\\
		&=B_iv\\
		&=T''_{w_\bullet,+1}(E_{\tau(i)}) K_{i}^{-1}v_{\eta}+l.o.t.\\
		&=q_i^{-\langle \alpha_i^\vee,\eta\rangle}T''_{w_\bullet,+1}(E_{\tau(i)}(T''_{w_\bullet,+1})^{-1}v_{\eta})+l.o.t.\\
		&=q_i^{-\langle \alpha_i^\vee,\eta\rangle}T''_{w_\bullet,+1}(E_{\tau(i)}v_{\eta})+l.o.t.\,.
	\end{align*}
	Applying $(T''_{w_\bullet,+1})^{-1}$ yields
	\begin{align*}
		(T''_{w_\bullet,+1})^{-1}\chi(B_i)v_\eta+l.o.t&=\chi(B_i)v_\eta+l.o.t.\\
		&=q_i^{-\langle \alpha_i^\vee,\eta\rangle}E_{\tau(i)}v_{\eta}+l.o.t..
	\end{align*}
	Again, because the weight of $E_{\tau(i)}v_\eta$ is greater than the weight of $v_\eta$ and because of maximality of $\eta$, this leads to $E_{\tau(i)}v_{\eta}=0$. We conclude that $E_iv_\eta=0$ for each $i\in \I$, which means that $\eta$ is equal to the highest weight $\lambda$. The remainder of the statements are proved analogously.
\end{proof}
If $\chi\in \widehat{\uqb}$, then we refer to the representation space of $\chi$ as $V_\chi\cong \F$.
\begin{lem}\label{lem:unique}
	Let $\lambda\in X^+$ be a dominant weight and $\chi\in \widehat{\uqbs}$ be a character. Then the modules $L(\lambda)$ and $L(\lambda)^\ast$ each contain at most one $1$-dimensional $\uqbs$ submodule of type $\chi$. 
\end{lem}
\begin{proof}
	Suppose that $v,w\in L(\lambda)$ span one-dimensional dimensional modules with the property that 
	$$b v=\chi(b)v\qquad\text{and}\qquad b w=\chi(b)w\qquad \text{for all }b\in\uqbs.$$
	As a result of Proposition \ref{prop:weight}, we may assume that $v=v_\lambda+l.o.t$ and $w=v_\lambda+l.o.t.$. The difference $v-w$ spans a one-dimensional module, yet the highest order term is not equal to a multiple of the highest weight vector. As a result of Proposition \ref{prop:weight}, this forces $w-v=0$. A similar argument for $L(\lambda)^\ast$ shows the desired conclusion.
\end{proof}
\begin{lem}\label{lem:uqbullet}
	Let $\lambda\in X^+$ be a dominant weight and suppose that $V_\chi$ is a one-dimensional $\uqbs$ submodule in $L(\lambda)$. Then $\chi|_{\uq_\bullet}=\epsilon$ and $\chi(K_h)=q^{\langle h,\lambda\rangle}$ for each $K_h\in \uq_\Theta^0.$
\end{lem}
\begin{proof}
	Let $\lambda\in X^+$ be a dominant weight and suppose that $v\in L(\lambda)$ spans the one-dimensional $\uqbs$ submodule $V_\chi$. By Proposition \ref{prop:weight}, we may assume that $v=v_\lambda+l.o.t.$. Let $K_h\in \uq_\Theta^0$. Then we have
	$$q^{\langle h,\lambda\rangle }v_\lambda+l.o.t=K_h v=\chi(K_h)v,$$
	which implies that $\chi(K_h)=q^{\langle h,\lambda\rangle }$. The only one-dimensional weight module of the quantum group $\uq_\bullet$ is the trivial module. Hence, the action of $\uq_\bullet$ must be equal to the action of $\epsilon$.
\end{proof}
\begin{lem}\label{lem:tech3}
	Let $\lambda\in X^+$ be a dominant weight and suppose that $V_\chi$ and $V_{\chi'}$ are one-dimensional $\uqbs$ submodules in $L(\lambda)$. Then it holds that $V_\chi=V_{\chi'}$ if and only if 
	$$\chi(B_i)=\chi'(B_i)\qquad\text{for all}\qquad i\in \I_\circ.$$
\end{lem}
\begin{proof}
	By Lemma \ref{lem:unique}, the multiplicity of a one-dimensional module is at most one. Using Lemma \ref{lem:uqbullet} we see that if $V_\chi$ and $V_{\chi'}$ are one-dimensional $\uqb$ submodules in $L(\lambda)$ we have $\chi|_{\uq_\bullet \uq_\Theta^0}=\chi'|_{\uq_\bullet \uq_\Theta^0}$. Henceforth, we have
	$$\chi(B_i)=\chi'(B_i)\quad \text{for all}\quad i\in \I_\circ\qquad \text{if and only if }\qquad V_\chi=V_{\chi'}.\qedhere$$
\end{proof}
This result will be one of our main tools for studying symmetries of characters. As a first consequence of Lemma \ref{lem:tech3}, we will show a relation between the bar involutions and characters. We first need a preliminary technical lemma that follows from \cite[\mbox{Prop 5.2}]{Mueller1999}.
\begin{lem}\label{lem:muller}
Let $\uqbs\subset \mathbf{U}_q(\mathfrak{sl}(2))$ be a coideal subalgebra of type $\mathsf{AI}_1$. Then for each $n\in \N$ eigenvalues of $B=F+cEK^{-1}+s K^{-1}\in \End (L(n\omega_1))$ are of the form $E(c,s,q,l)=\frac{(q^l-q^{-l})}{2}\sqrt{s^2+\frac{4qc}{(q-q^{-1})^2}}+\frac{s}{2}(q^{l/2}-q^{-l/2})^2+s$, where $l\in \Z$.
\end{lem}
\begin{proof}
According to \cite[\mbox{Prop 5.2}]{Mueller1999} the eigenvalues of $D=F+cEK^{-1}+s (K^{-1}-1)\in \End (L(n\omega_1))$ are of the form $\frac{s}{2}(q^{l/2}-q^{-l/2})^2+\frac{1}{2}(q^l-q^{-l})R$, with $l\in \Z$ and where $R$ is fixed and satisfies the equation
$$R^2=s^2+\frac{4cq}{(q-q^{-1})^2}.$$
We deduce that the eigenvalues of $B$ are of the form \begin{equation}\label{eq:eigenvalues}
E(c,s,q,l)=\frac{(q^l-q^{-l})}{2}\sqrt{s^2+\frac{4qc}{(q-q^{-1})^2}}+\frac{s}{2}(q^{l/2}-q^{-l/2})^2+s
\end{equation}
with $l\in \Z$.
\end{proof}
\begin{thm}\label{thm:unique}
	Let $\boldsymbol{c}$ be a uniform parameter. Then for each $\chi\in \widehat{\uqb}$ that occurs in a finite dimensional $\uq$ weight module it holds that
	$$\overline{\chi( \overline{b}^{\uqb})}=\chi(b)\qquad \text{for all}\qquad b\in \uqb.$$
\end{thm}
\begin{proof}
	Let $\lambda\in X^+$ be a weight such that $V_\chi$ occurs in $L(\lambda)$, and let $v\in L(\lambda)$ be a vector spanning $V_\chi$. Because $\boldsymbol{c}$ is uniform, see Theorem \ref{thm:barinvolution}, this gives us the existence of the $\imath$-bar involution. The intertwining properties of the quasi-$K$ matrix imply that the vector $\Upsilon\overline{v}$ spans a one-dimensional $\uqb$ module. For $b\in\uqb$, we have
	\begin{equation}\label{eq:chibar}
		b \Upsilon\overline{v}=\Upsilon\overline{\overline{b}^{\uqb}}\overline{v}=\overline{\chi(\overline{b}^{\uqb})}\Upsilon\overline{v}.
	\end{equation}
	By Lemma \ref{lem:tech3}, we have
	$$\mathrm{span}_\F\{v\}=\mathrm{span}_\F\{\Upsilon\overline{v}\}\qquad \text{if and only if }\qquad\chi(B_i)=\overline{\chi( \overline{B_i}^{\uqb})}\quad \text{for all}\quad i\in \I_\circ .$$
	If $i\notin \I_{\mathrm{ns}}$ then it follows by \cite[\mbox{Prop 10.1}]{Kolb2023} that
	$\overline{\chi( B_i)}=0=\chi(B_i).$ Because $\overline{B_i}^{\uqb}=B_i$, this shows the desired claim. 
	Hence, it remains to show that $\overline{\chi( \overline{B_i}^{\uqb})}=\chi(B_i)$ in the case that $i\in \I_{\mathrm{ns}}$. Consider the rank one Satake sub-diagram induced by the index $i$. By definition, $\tau(i)=i$ and $c_{ij}=0$ for all $j\in \I_\bullet$. This means that the rank one sub-diagram is of type $\mathsf{AI_1}$. 
	Using Lemma \ref{lem:muller} we see that there exists a $l\in \Z$ such that $\chi( B_i)=[l]_{q_i}\sqrt{q_ic_i}$. Using that the parameter $\boldsymbol{c}$ is uniform, we deduce that
	$\overline{\chi( B_i)}=\overline{[l]_{q_i}\sqrt{c_iq_i}}=[l]_{q_i}\overline{c_i^{1/2}}q_i^{-1/2}\stackrel{{ (\ref{eq:uniform})}}{=}[l]_{q_i}\sqrt{c_iq_i}=\chi(B_i).$ Again, because $\overline{B_i}^{\uqb}=B_i$, this shows the desired claim. 
\end{proof}
We apply Theorem \ref{thm:unique} to show that, up to a scalar multiple, one-dimensional $\uqb$ modules are invariant under the $\imath$-bar involution.
\begin{lem}\label{lem:upsion}
	Let $\boldsymbol{c}$ be an uniform parameter, let $\chi \in \widehat{\uqb}$ a character, and let $\lambda$ a dominant integral weight. If $v\in L(\lambda)$ and $f\in L(\lambda)^\ast$ are $\chi$-spherical vectors, then there exist nonzero $a,a^r\in \F$ such that $\Upsilon \overline{av}=av $ and $\overline{a^rf}\Upsilon=a^rf$.
\end{lem}
\begin{proof}
	By Theorem \ref{prop:weight} there exists a nonzero scalar $a$ such that $av=v_{w_0\lambda}+h.o.t.$.
	Using (\ref{eq:chibar}), the vector $\Upsilon\overline{av}$ spans a one-dimensional $\uqb$ module. The action of $\uqb$ on the module is given by $b\mapsto \overline{\chi(\overline{b}^{\uqb})}$ for $b\in \uqb$. Moreover, Theorem \ref{thm:unique} implies that $\overline{\chi(\overline{b}^{\uqb})}=\chi(b)$ for each $b\in \uqb$. Therefore, Lemma \ref{lem:unique} gives the existence of a scalar $b\in \F$ such that $\Upsilon\overline{av}=bav$. Recall that the quasi-$K$ matrix satisfies $\Upsilon=\sum_{\mu\in X}\Upsilon_\mu$, where $\Upsilon_\mu\in {\uq}_\mu^+$ and $\Upsilon_0=1$, cf. Theorem \ref{thm:intertwiner}.  As $\overline{v_{w_0\lambda}}=v_{w_0\lambda}$, this implies that $\Upsilon\overline{av}=v_{w_0\lambda}+h.o.t.$. So, we may conclude that $b=1$.
	An analogous argument gives the existence of $a^r\in \F$ with $\overline{a^rf}\Upsilon=a^rf$.
\end{proof}
\subsection{Spherical functions of type $\boldmath{\chi}$}\label{sec:sph}
The goal of this section is to introduce the notion of $\chi$-spherical functions and study their structure. 
\begin{rem}\label{rem:asumpar}
	\begin{enumerate}[(i)]
		\item  In this paper we assume that the parameter $(\boldsymbol{c},\boldsymbol{s})$ is chosen such that each module $L(\lambda)$ is completely reducible as $\uqbs$-module.
		\item Let $\uqb$ be a quantum symmetric pair coideal subalgebra where $\boldsymbol{c}$ is uniform. According to \cite[\mbox{Cor 3.2.4}]{Watanabe2023} each finite dimensional $\uq$ weight module is completely reducible as $\uqb$ module. Via the isomorphism $\Phi_{\boldsymbol{a}}$ one obtains complete reducibility for more general parameters. 
	\end{enumerate}
\end{rem}
Both Watanabe \cite[\mbox{Cor 3.2.4}]{Watanabe2023}  and Letzter \cite[\mbox{Thm 3.3}]{Letzter2000} obtain complete reducibility within their respective frameworks. Their methods rely on the existence of a non-degenerate bilinear form, but the forms are contravariant with respect to different automorphisms.
\begin{defi}[Spherical functions]
	Let $\uqbs,\uqds\subset \uq$ be quantum symmetric pair coideal subalgebras and let $\chi$ and $\eta$ be characters of $\uqbs$ and $\uqds$, respectively. We say that $\varphi\in \qfa$ is a\textit{ $(\chi,\eta)$-spherical function} if $\varphi$ transforms as
	$$\varphi(bxb')=\chi(b)\varphi(x)\eta(b')\qquad\text{for all }  b\in \uqbs,b'\in \uqds,\text{ and } x\in \uq.$$
	If $\uqbs=\uqds$, then we refer to $(\chi,\chi)$-spherical functions as $\chi$\textit{-spherical functions.}
	We will refer to $(\epsilon,\epsilon)$-spherical functions as \textit{zonal spherical functions.}
	We write $^{\uqds}_{\chi'}\qfa_{\chi}^{\uqbs}\subset \qfa$ to denote the the $\F$-vector space of $(\chi',\chi)$-spherical functions for $(\uqds,\uqbs)$. If $\uqbs=\uqds$ and $\chi'=\chi$, then we write $\qfachi:={^{\uqbs}_{\chi}\qfa_{\chi}^{\uqbs}}$.
\end{defi}
\begin{prop}\label{prop:prop1}
	Let $\chi \in\widehat{\uqbs}$ be a character, and let $\lambda\in X^+$ be a dominant weight. Then $L(\lambda)$ contains a one-dimensional module of type $\chi$ if and only if $L(\lambda)^\ast$ contains a one-dimensional module of type $\chi$.
\end{prop}
\begin{proof}
	Suppose that $L(\lambda)$ contains a one-dimensional module $V_\chi=\text{span}_\F\{v\}$ of type $\chi$. By complete reducibility of $L(\lambda)$, as $\uqbs$ module, there exists a $\uqbs$ invariant complement $W$. Define $v^\ast\in L(\lambda)^\ast $ by $v^\ast (av)=a$ for $a\in\F$ and $v^\ast(w)=0$ for $w\in W.$ Then we have
	$$(v^\ast b)(v)=v^\ast (bv)=\chi(b)v^\ast(v),\quad\text{and}\quad ( v^\ast b)(w)=v^\ast(bw)=0\quad \text{for}\quad w\in W, b\in \uqb.$$
	Hence, $v^\ast$ is a $\chi$-spherical vector. Conversely, let $v^\ast\in L(\lambda)^\ast$ be a $\chi$-spherical vector. By complete reducibility there exists a $v\in L(\lambda)$ with $v^\ast(v)=1$ that spans an invariant complement of $\ker(v^\ast)$. 
	As
	$$\chi(b)v^\ast (v)=( v^\ast b)(v)=v^\ast (bv)\qquad \text{for each}\qquad b\in\uqbs,$$
	it follows that $v$ is a $\chi$-spherical vector in $L(\lambda)$.
\end{proof}
The vectors $v^\ast\in L(\lambda)^\ast$ and $v\in L(\lambda)$ are refered to as \textit{spherical vectors}, or $\chi$-\textit{spherical vectors}. 
\begin{thm}\label{thm:unique2}
	Let $\uqbs$ and $\uqds$ be quantum symmetric pair coideal subalgebras and let $\chi$ be a character of $\uqbs$ and $\eta$ be a character of $\uqds$. Then for each dominant weight $\lambda\in X^+$ for which, both the multiplicity $\chi$ in $L(\lambda)$ and the multiplicity $\eta$ in $L(\lambda)$ are $1$, there exists a, up to nonzero scalar multiple, unique spherical function of type $(\chi,\eta)$ of the form $\varphi\in L(\lambda)^\ast\otimes L(\lambda)$.
\end{thm}
\begin{proof}
	Existence follows by Proposition \ref{prop:prop1}.	Uniqueness, up to a scalar multiple, is a result of Lemma \ref{lem:unique}.
\end{proof}
In the following lemma we study realizations of spherical functions and spherical vectors. The identifications made will be a crucial tool for studying the relation between spherical functions for different parameters. 
Recall the isomorphism $\Phi_{\boldsymbol{a}}$ from Section \ref{sec:iota}. In the next lemma, we analyze the modules $^{\Phi_{\boldsymbol{a}}}L(\lambda)$ and $^{\Phi_{\boldsymbol{a}}}L(\lambda)^\ast$. Let $\lambda\in X^+$ be a dominant weight and $\boldsymbol{a}\in (\F^\times )^{\I}$. Set for each weight $\mu=\lambda-\eta$, with $\eta=\sum_{i\in\I}t_i\alpha_i$ and $t_i\in \Z_{\geq0}$,
$$\phi_{\boldsymbol{a}}^{L(\lambda)}(\eta):=\prod_{i\in\I}a_i^{t_i/2}.$$
\begin{lem}\label{lem:tech4} Let $\lambda\in X^+$ be a dominant weight and $\boldsymbol{a}\in (\F^\times)^\I$, with $a_j=1$ for each $j\in \I_\bullet$. Then the following four statements hold.
	\begin{enumerate}[(i)]
		\item The modules $L(\lambda)$ and $^{\Phi_{\boldsymbol{a}}}L(\lambda)$ are isomorphic, via $\Phi_{\boldsymbol{a}}:v_\mu\mapsto \phi^{L(\lambda)}_{\boldsymbol{a}}(\mu)v_\mu$ for weight vectors $v_\mu\in L(\lambda)_\mu$.
		\item The modules $L(\lambda)^\ast$ and $\big(^{\Phi_{\boldsymbol{a}}}L(\lambda)\big)^\ast$ are isomorphic, via $(\Phi^\ast_{\boldsymbol{a}})^{-1}:f_\mu\mapsto (\phi^{L(\lambda)}_{\boldsymbol{a}}(\mu))^{-1}f_\mu$ for weight vectors $f_\mu\in L(\lambda)^\ast_\mu$.
		\item For each parameter $(\boldsymbol{c},\boldsymbol{s})$ it holds that $\Phi_{\boldsymbol{a}}(\uqbs)=\uqa$, where we define $\Phi_{a}(\boldsymbol{c})_i=a_i^{1/2}a_{ \tau(i)}^{1/2}c_i$ and $\Phi_{a}(\boldsymbol{s})_i=a_i^{1/2}s_i,$  for $i\in \I_\circ$.
		\item  If $c^{L(\lambda)}_{f,v}$ is a $(\chi,\eta)$-spherical function for $(\uqbs, \uqbt)$, then $c^{L(\lambda)}_{\Phi^\ast_{\boldsymbol{a}}(f),\Phi_{\boldsymbol{a}}(v)}$ is a $(\chi\circ \Phi_{\boldsymbol{a}}^{-1},\eta\circ \Phi_{\boldsymbol{a}})$-spherical function for $(\uqa,\uqat)$.
	\end{enumerate}
\end{lem}
\begin{proof}
	\begin{enumerate}[(i)]
		\item The modules $L(\lambda)$ and $^{\Phi_{\boldsymbol{a}}}L(\lambda)$ are isomorphic, as both are irreducible highest weight modules of weight $\lambda$. Note than we can write each weight vector $v_\mu\in L(\lambda)_\mu$ as a linear combination of $F_Jv_\lambda$, where $J$ is a finite sequence of roots such that $F_Jv_\lambda\in L(\lambda)_\mu$. Any $\uq$ module isomorphism $L(\lambda)\to\, ^{\Phi_{\boldsymbol{a}}}L(\lambda)$ sends $v_\lambda\in L(\lambda)$ to a scalar multiple of $v_\lambda \in\, ^{\Phi_{\boldsymbol{a}}}L(\lambda)$. We normalize this scalar to equal $1$. In this case the vector $F_Jv_\lambda$ is mapped to 
		$$\Phi_{\boldsymbol{a}}(F_J)v_\lambda=a_{i_1}F_{i_1}\dots a_{i_n}F_{i_n}=\phi^{L(\lambda)}_{\boldsymbol{a}}(\mu)F_Jv_\lambda.$$
		\item Analogous to part $(i)$.
		\item By definition, $\Phi_{\boldsymbol{a}}$ is the identity on $\uq_\bullet\uq_{\Theta}^0$. For $i\in\I_\circ$ we have
		\begin{align*}
			F_i+c_i T''_{w_\bullet,+1}(E_{\tau(i)})K_i^{-1}+s_iK_i^{-1}
			\mapsto a_i^{-1/2}F_i+a_{ \tau(i)}^{1/2}c_i T''_{w_\bullet,+1}(E_{\tau(i)})K_i^{-1}+s_iK_i^{-1}\\
			=a_i^{-1/2}\bigg(F_i+(a_ia_{ \tau(i)})^{1/2}c_i T''_{w_\bullet,+1}(E_{\tau(i)})K_i^{-1}+a_i^{1/2}s_iK_i^{-1}\bigg)\in \uqa.
		\end{align*}
		\item This follows by part $(iii)$.\qedhere
	\end{enumerate}
\end{proof}
\begin{rem}\label{lem:tech}
	For $h\in \frac{1}{2}Y$ the action $\mathrm{ad}(K_h)$ on $\uq$ and the action of $K_h$ on $L(\lambda)$ may be identified with $\Phi_{\boldsymbol{a}}$, with $a_i=q^{\langle 2h,\alpha_i\rangle}$ for $i\in\I_\circ$. Henceforth, Lemma \ref{lem:tech4} also applies to $\mathrm{ad}(K_h)(\uqbs)$.
\end{rem}
In the following lemma we analyze the relationship between the matrix coefficients $c^{L(\lambda)}_{f,v}$ and $c^{L(\lambda)}_{\Phi^\ast_{\boldsymbol{a}}(f),\Phi_{\boldsymbol{a}}(v)}$ under restriction to the subalgebra $\uq^0$.
\begin{lem}\label{lem:transformation}
Let $\boldsymbol{a}\in (\F^\times)^{\I}$ and let  $c^{L(\lambda)}_{f,v}\in L (\lambda)^\ast\otimes L(\lambda)$ be a matrix entry. Then for each $h\in Y$ it holds that 
$c^{L(\lambda)}_{f,v}(K_h)=c^{L(\lambda)}_{\Phi^\ast_{\boldsymbol{a}}(f),\Phi_{\boldsymbol{a}}(v)}(K_h).$
\end{lem}
\begin{proof}
Write $f=\sum_{\mu\in X} f_\mu$ and $v=\sum_{\mu\in X} v_\mu$, where $f_\mu\in L(\lambda)^\ast_\mu$ and $v_\mu\in L(\lambda)_\mu$. Let $h\in Y$. Then it holds that
	$$c^{L(\lambda)}_{f,v}(K_h)=\sum_{\mu\in X} c^{L(\lambda)}_{f_\mu,v_\mu}(K_h),$$
	because $K_h$ acts diagonally on the weight spaces and different weight spaces are orthogonal. Let $h\in Y$. Then it holds that \belowdisplayskip=-12pt
	\begin{align*}
		c^{L(\lambda)}_{\Phi^\ast_{\boldsymbol{a}}(f),\Phi_{\boldsymbol{a}}(v)}(K_h)&=\sum_{\mu \in X} c^{L(\lambda)}_{\phi^{L(\lambda)}_{\boldsymbol{a}}(\mu)^{-1}f_\mu,\phi^{L(\lambda)}_{\boldsymbol{a}}(\mu) v_\mu}(K_h)\\
		&=\sum_{\mu \in X} c^{L(\lambda)}_{ f_\mu,v_\mu}(K_h)\\
		&=c^{L(\lambda)}_{ f,v}(K_h).
	\end{align*}\qedhere
\end{proof}
\subsection{The existence of characters}\label{sec:limit}
In this section we study the existence of one-dimensional $\uqbs$ modules, particularly those that arise as submodules of $\uq$ highest weight modules. Using a specialization argument we reduce this problem to the existence of one-dimensional $\mathbf{U}(\mathfrak{k})$ modules  that arise as submodules of $\mathbf{U}(\mathfrak{g})$ highest weight modules. The latter are characterized in terms of weights in \cite[\mbox{Thm 7.2}]{Schlichtkrull1984}. \\
\\
To form the specialization argument, we recall the specialization of $\uqbs$ and its modules. Let $\A$ be the smallest subring of $\F$ containing $q$ such that every element in $\A$ that is not contained in the ideal generated by $(q - 1)$ has a square root and is invertible, cf. \cite[\mbox{Sec 1}]{Letzter2000}. Let $\wideparen{\mathbf{U}}_q(\mathfrak{g})$ denote the $\A$ subalgebra generated by $E_i,F_i,K_i$ and $\frac{K_i-1}{q-1}$, with $i\in \I$. 
Recall that  $\wideparen{\mathbf{U}}_q(\mathfrak{g})/(q-1)\wideparen{\mathbf{U}}_q(\mathfrak{g})$ is isomorphic to $\mathbf{U}(\mathfrak{g})$, cf. \cite[\mbox{6.11}]{Joseph1994}. For a dominant weight $\lambda\in X^+$ the $\A$ module $L_{\A}(\lambda):=\wideparen{\mathbf{U}}_q(\mathfrak{g})v_\lambda$ is specialized by setting $L(\lambda)^1:=L_{\A}(\lambda)/(q-1)L_{\A}(\lambda)$. The $\mathbf{U}(\mathfrak{g})$ module $L(\lambda)^1$ is the highest weight module of weight $\lambda$.  We write $\cl$ for the quotient maps $\wideparen{\mathbf{U}}_q(\mathfrak{g})\to \mathbf{U}(\mathfrak{g})$, $L^\A(\lambda)\to L(\lambda)^1$ and $\A\to \C$. If $B\subset \uq$ is a subalgebra, we use $\wideparen{B}$ to denote the intersection $B\cap \wideparen{\mathbf{U}}_q(\mathfrak{g})$. A $B$ module $V$ is \textit{specializable} if there exists a basis $\mathcal B=\{v_1,\dots v_n\}\subset V$ such that $\wideparen{B}\,\cdot\mathrm{span}_{\A}\mathcal B\subset \mathrm{span}_{\A}\mathcal B$.
A parameter $(\boldsymbol{c},\boldsymbol{s})$ is  \textit{specializable} if $c_i,s_i\in \A$ and $\cl(\frac{c_i}{s(\tau(i))})=-1$ for each $i\in \I_\circ$, where $s:\I\to \F^\times$ is defined as in \cite[\mbox{Eq (3.2)}]{Balagovic2015}. In this case we have $\mathbf{U}(\mathfrak{g})\supset \wideparen{\uqbs}/(1-q)\wideparen{\uqbs}=\mathbf{U}(\mathfrak{k})$, cf. \cite[\mbox{Thm 4.9}]{Letzter1999} or \cite[\mbox{Thm 10.8}]{Kolb2014}. This allows us to form the $\mathbf{U}(\mathfrak{k})$ module $\widetilde{V}:= \sum_ {v\in\mathcal B}\A v\otimes _{\A}\C$.
In the following lemma we use specialization results from \cite{Letzter2000}. However, in \cite{Letzter2000}, a different convention of the coideal subalgebra is used. To be precise, Letzter uses left coideal subalgebras $\uqbl$ instead of right coideal subalgebras. The algebras $\uqbs$ and $\uqbl$ are isomorphic via the isomorphism $\omega$, that sends the Hopf structure to the opposite Hopf structure, cf. \cite[\mbox{Sec 2}]{Letzter2000}. Moreover, Letzter works with quantum groups over the quotient field of $\A$, which is isomorphic to $\F$.
\begin{thm}\label{lem:onedim} 
	Let $\lambda\in X^+$ be a dominant weight, and let $(\boldsymbol{c},\boldsymbol{s})$ be a specializable parameter. If the $\mathbf{U}_q(\mathfrak{g})$ module $L(\lambda)$ decomposes as $\bigoplus_{i=1}^m L_i$ into simple $\uqbs$ modules, then the $\mathbf{U}(\mathfrak{g})$ module $L(\lambda)^1$ decomposes as $\bigoplus_{i=1}^m \widetilde{L}_i$ into simple $\mathbf{U}(\mathfrak{k})$ modules.
\end{thm}
\begin{proof}
	Let $L(\lambda)=\bigoplus_{i=1}^mL_i$ denote the decomposition as simple $\uqbl$ modules. According to the proof of \cite[\mbox{Prop 7.12}]{Letzter2000}, there exists a specializable basis $\mathcal B(i):=\{u^i_1v_\lambda,\dots, u^i_{n_i}v_\lambda\}$ of $L_i$, where $u^i_j\in \wideparen{\mathbf{U}}_q(\mathfrak{g})$ for each $1\leq i \leq m$ and $1\leq j\leq n_i$. If needed, we may rescale $u^i_j$ such that $\cl (u^i_j)\neq 0$ for each $1\leq i \leq m$ and $1\leq j\leq n_i$. For each $1\leq i \leq m$, consider the $\A$ module $L_i^\A=L_i\cap L^\A(\lambda)$, we note that $L_i^\A$ is a free $\A$ module by \cite[\mbox{Lem 7.3}]{Letzter2000} and $L^\A(\lambda)=\bigoplus _{i=1}^m L^\A_i$. For each $1\leq i\leq m$ the $\A$ module $M_i:=\text{span}_{\A} \mathcal B(i)$ is contained in $L^\A_i$. We claim that the following equality holds:$$\dim _\C(L_i^\A/ (q-1)L_i^\A(\lambda))=\dim_{\C}(M_i/(1-q)M_i)=\dim_{\C}(M_i/(1-q)L_i^\A)=\dim_\F (L_i).$$
	The equality $\dim_\F(L_i)=\dim _\C(L_i^\A/ (q-1)L_i^\A(\lambda))$ follows from \cite[\mbox{Lem 7.3}]{Letzter2000} and an dimension counting argument. The equality $\dim_\F(L_i)=\dim_{\C}(M_i/(1-q)M_i)$ is a consequence of \cite[\mbox{Thm 7.6}]{Letzter2000}. Lastly, the equality $\dim_{\C}(M_i/(1-q)M_i)=\dim_{\C}(M_i/(1-q)L_i^\A)$ holds as $\A$ is a principal ideal domain and $\cl (u^i_j)\neq 0$ for all $1\leq j\leq n_i$.
	Consider the well-defined $\C$-linear map
	$$T_1:M_i/(1-q)M_i\to  M_i/(1-q)L_i^\A,\qquad x+ (1-q)M_i\mapsto x+(1-q)L_i^\A(\lambda),\qquad x\in M_i.$$
	As the dimension of the domain and co-domain are equal and $T_1$ is surjective, $T_1$ is an isomorphism.
	Analogously the well-defined $\C$-linear map
	$$T_2:M_i/(1-q)L_i^\A(\lambda)\to  L_i^\A(\lambda)/(1-q)L_i^\A,\qquad x+ (1-q)M_i\mapsto x+(1-q)L_i^\A(\lambda),$$
	where $x\in M_i$, is an isomorphism because $T_2$ is injective. As $\wideparen{\uqbl}$ preserves the $\A$ modules $(1-q)M_i^\A$ and $(1-q)L_i^\A$, the composition $T_2\circ T_1$ is a $\mathbf{U}(\mathfrak{k})$ module isomorphism. This shows that the $\mathbf{U}(\mathfrak{g})$ module $L(\lambda)^1$ has a decomposition 
	$$L(\lambda)^1\cong\bigoplus_{i=1}^m L_i^\A/(1-q)L_i^\A,$$
	where by \cite[\mbox{Prop 7.6}]{Letzter2000} each constituent $L_i^\A/(1-q)L_i^\A\cong M_i/(q-1)M_i\cong \widetilde{M_i}$ is a simple $\mathbf{U}(\mathfrak{k})$ module. This gives the result for $\wideparen{\uqbl}$.
%
	The desired conclusion for $\uqbs$ is obtained by twisting $L(\lambda)$ with the automorphism $\omega$.
\end{proof}
Let $V_\chi\subset L(\lambda)$ be a one-dimensional $\uqbs$ module of type $\chi$ with specializable basis $\{v\}$. Note that the $\mathbf{U}(\mathfrak{k})$ module $\widetilde{V_\chi}= \A v\otimes_{\A}\C$ is one-dimensional, we write $\cl(\chi)$ for the corresponding character. 
\begin{lem}\label{lem:clchar}
Let $\lambda\in X^+$ be a dominant weight and let $V_\chi\subset L(\lambda)$ be a one-dimensional $\uqbs$ module. Then, for each $i\in \I_{\mathrm{ns}}$ there exists an integer $l_i\in \Z$ such that	\begin{align}\label{eq:form1}
	&	\chi(B_i)=E(c_i,s_i,q_i,l_i)\quad \text{if }\in \I_{\mathrm{ns}}&& 
	\chi(B_i)=0\quad \text{if }i\notin \I_{\mathrm{ns}}\\
	&\chi|_{\uq_\bullet}=\epsilon&& 
		\chi(K_h)=q^{\langle h,\lambda\rangle} \quad \text{if }\theta(h)=h \nonumber.		
	\end{align}
	Here, $E(c_i,s_i,q_i,l_i)\in \F$ is the element defined in $(\ref{eq:eigenvalues})$.
	Furthermore, if $(\boldsymbol{c},\boldsymbol{s})$ is a specializable parameter, then $\cl(\chi)$ is the unique character of $\mathbf{U}(\mathfrak{k})$ that satisfies
	\begin{align}\label{eq:findimuq}
				&\cl(\chi)(\cl(B_i)))= l_i\sqrt{-1}+\cl(s_i)\quad \text{if }\in \I_{\mathrm{ns}}&&\cl(\chi)(\cl(B_i))=0\quad \text{if }i\notin \I_{\mathrm{ns}}\\
		&\cl (\chi)_{\mathbf{U}(\mathfrak{g})_\bullet}=\epsilon&& \cl(\chi)(h)=\langle h,\lambda\rangle \quad \text{if }\theta(h)=h\nonumber.
	\end{align}
\end{lem}
\begin{proof}
Using \cite[\mbox{Prop 10.1}]{Kolb2023}, Lemma \ref{lem:uqbullet} and Lemma \ref{lem:muller} the character $\chi$ takes the form (\ref{eq:form1}). The action of $\mathbf{U}(\mathfrak{k})$ on $\widetilde{V_\chi}$ follows directly.
\end{proof}
\begin{rem}\label{rem:char}
Analogously to the proof of Lemma \ref{lem:clchar}, any character of $\mathbf{U}(\mathfrak{k})$ that occurs in a finite dimensional $\mathbf{U}(\mathfrak{g})$ weight module takes the form (\ref{eq:findimuq}). 
\end{rem}
\begin{rem}
Let $(\boldsymbol{c},\boldsymbol{s})$ be a specializable parameter and let $\chi$ be a character of $\uqbs$ occurring in a simple module $L(\lambda)$. As a result of the presentation of Lemma \ref{lem:clchar} it is automatic that $\chi: \wideparen{\uqbs}\to \A$.
\end{rem}
\begin{thm}\label{thm:q=1}
	Let $(\boldsymbol{c},\boldsymbol{s})$ be a specializable parameter and let $\chi\in\widehat{\uqbs}$ be a character. Then the following two statements are equivalent:
	\begin{enumerate}[(i)]
		\item  The $\uq$ module $L(\lambda)$ contains a module of type $\chi$ for $\uqbs$.
		\item The $\mathbf{U}(\mathfrak{g})$ module $L(\lambda)^1$ contains a module of type $\cl(\chi)$ for $\mathbf{U}(\mathfrak{k})$.
	\end{enumerate}
\end{thm}
\begin{proof}
	The implication $(i)\implies(ii)$ follows from Lemma \ref{lem:clchar} and Theorem \ref{lem:onedim}.
	Conversely, suppose that the $\mathbf{U}(\mathfrak{k})$ module $L(\lambda)^1$ contains a one-dimensional module $\widetilde{L}$ of type $\cl(\chi)$. According to Remark \ref{rem:char}  and Theorem \ref{lem:onedim} there exists a one-dimensional $\uqbs$ module $V_{\chi'}\subset L(\lambda)$ that specializes to $\widetilde{L}$. 
	To determine the character $\chi'$ we use Lemma \ref{lem:clchar}. Accordingly, we remain to determine the integer $k_i$ that satisfies $\chi(B_i)=E(c_i,s_i,q_i,k_i)$ for $i\in \I_{\mathrm{ns}}$. Recall that the parameter $(\boldsymbol{c},\boldsymbol{s})$ is specializable, which in particular means that $\cl(c_i)=-1$ as $s(i)=1$, cf. \cite[\mbox{Eq (3.2)}]{Kolb2014}. Hence, the equality $$ l_i\sqrt{-1}=\cl (\chi')(\cl(B_i))-\cl (s_i)=\cl(E(c_i,s_i,q_i,k_i))-\cl (s_i)=k_i\sqrt{-1}$$
	shows that $l_i=k_i$.
\end{proof}
To extend Theorem \ref{thm:q=1} to more general parameters we use Lemma \ref{lem:tech4}.
\begin{cor}\label{cor:hermetian}
	Let $(\boldsymbol{c},\boldsymbol{s})$ be a specializable parameter, let $\chi$ be a character of $\uqbs$ and let $\boldsymbol{a}\in ({\F^\times})^{\I}$ with $a_i=1$ for each $i\in \I_\bullet$. Then the following two statements are equivalent
	\begin{enumerate}[(i)]
		\item  The $\uq$ module $L(\lambda)$ contains a module of type $\chi\circ \Phi_{\boldsymbol{a}}^{-1}$ for $\uqa$.
		\item The $\mathbf{U}(\mathfrak{g})$ module $L(\lambda)^1$ contains a module of type $\cl(\chi)$ for $\mathbf{U}(\mathfrak{k})$.
	\end{enumerate}
\end{cor}
Using Corollary \ref{cor:hermetian} we relate the existence of non-trivial characters to $(\mathfrak{g},\mathfrak{k})$ being Hermitian. A symmetric pair $(\mathfrak{g},\mathfrak{k})$ is called \textit{irreducible} if $\mathfrak{g}$ cannot be written as the
direct sum of two semisimple Lie subalgebras which both admit $\Theta$ as an involution. An irreducible symmetric pair $(\mathfrak{g},\mathfrak{k})$ is said to be of \textit{Hermitian type} if $\mathfrak{k}$ has a non-trivial center, or according to \cite[\mbox{Ch 5}]{Heckman1995} this is equivalent to the existence of non-trivial one-dimensional $\mathbf{U}(\mathfrak{k})$ modules in finite dimensional $\mathbf{U}(\mathfrak{g})$ weight modules.
\begin{prop}\label{cor:herm2}
 	Let $(\boldsymbol{c},\boldsymbol{s})$ be a specializable parameter, let $\boldsymbol{a}\in ({\F^\times})^{\I}$ with $a_i=1$ for each $i\in \I_\bullet$ and let the symmetric pair $(\mathfrak{g},\mathfrak{k})$ be irreducible. Then the following three statements are equivalent:
 \begin{enumerate}[(i)]
 	\item There exist countably infinite many pairwise non-isomorphic one-dimensional $\uqa$ modules that occur in finite dimensional $\uq$ weight modules.
 	\item There exist non-trivial one-dimensional $\uqa$ modules that occur in finite dimensional $\uq$ weight modules.
 	\item The symmetric pair $(\mathfrak{g},\mathfrak{k})$ is of Hermitian type.
 \end{enumerate}
\end{prop}
\begin{proof}
The implication $(i)\implies (ii)$ is clear.
Suppose that there exist a non-trivial one-dimensional $\uqa$ module $V_\chi \subset M$, where $M$ is a finite dimensional $\uq$ weight module. 
According to Corollary \ref{cor:hermetian} there exists a non-trivial one-dimensional $\mathbf{U}(\mathfrak{k})$ module in $\widetilde{M}$. 
So, $(\mathfrak{g},\mathfrak{k})$ is of Hermitian type. Lastly, suppose that $(\mathfrak{g},\mathfrak{k})$ is of Hermitian type. Then there are countably infinite many pairwise non-isomorphic non-trivial one-dimensional $\mathbf{U}(\mathfrak{k})$ modules that occur in finite dimensional $\mathbf{U}(\mathfrak{g})$ weight modules, cf. \cite[\mbox{Ch 5}]{Heckman1995}. By Theorem \ref{thm:q=1} these modules are in one-to-one correspondence with the non-trivial one-dimensional $\uqbs$ modules that occur in $\uq$ weight modules. Therefore, by Corollary \ref{cor:hermetian} there are countably infinite many pairwise non-isomorphic one-dimensional $\uqa$ modules that occur in finite dimensional $\uq$ weight modules.
\end{proof}
\begin{rem}\label{rem:herm}
Proposition \ref{cor:herm2} allows for an extension to non-specializable standard parameters, using results of \cite{Watanabe2024} and \cite{Watanabe2021}. The only standard cases where Proposition \ref{cor:herm2} does not apply in full generality occur when the Satake diagram contains a rank-one subdiagram of type $\mathsf{AIV}$. These diagrams correspond to the irreducible Hermitian types $\mathsf{AIV}, \mathsf{AIII}, \mathsf{DIII}, \mathsf{EIII}$. Let $j \in \I_\circ$ be an index such that the local type of this rank-one subdiagram is $\mathsf{AIV}$. Using Proposition \ref{cor:herm2}, together with the abstract isomorphism of quantum symmetric pair coideal subalgebras established in \cite[Lem 2.5.1]{Watanabe2021}, we conclude that for each $l \in \Z$, the map $\chi_l: \uqbs \to \F$ defined by
\begin{equation}\label{eq:characterwet}
\chi|_{\uq_\bullet}=\epsilon,\quad B_k\mapsto 0, \quad K_iK_{\tau(i)}^{-1}\mapsto 1,\quad K_jK_{\tau(j)}^{-1}\mapsto q^l,\quad k\in \I_\circ, j\neq i\in\I_\circ 
\end{equation}
is a character of $\uqbs$. The fact that the character $\chi_l$ is well defined may also be observed by the presentation of $\uqbs$ in terms of generators and relations, cf. \cite[Thm 7.1]{Kolb2014}. We note that the generators $B_k$, $E_j$ and $F_j$ for $k\in \I_\circ $ and $j\in \I_\bullet$ act locally nilpotent on $V_\chi$ and the generators of $\uq^0_{\Theta}$ act by a power of $q$. By \cite[\mbox{Prop 4.3.1}]{Watanabe2024} this implies that the $\uqbs$ module $V_\chi$ is \textit{integrable}, cf. \cite[Def. 3.3.4]{Watanabe2024}. Because $V_\chi$ is integrable, \cite[Def. 3.3.4]{Watanabe2024} implies that there exist dominant weights $\lambda,\mu\in X^+$ and a surjective $\uqbs$ module homomorphism $V^\imath(\lambda,\mu)\to V_\chi$. Here, $V^\imath(\lambda,\mu)$ is equal to to the $\uq$ module in $L(\mu)\otimes L(\lambda)$ generated by $ v_{w_\bullet\lambda}\otimes v_\mu$, where $v_{w_\bullet \lambda}$ is a distinguished weight vector in $L(\lambda)$ of weight $w_\bullet \lambda$. In particular, if $V^\imath(\lambda,\mu)$ is completely reducible as a $\uqbs$ module, $V_\chi$ embeds into the finite dimensional $\uq$ weight module $L(\lambda) \otimes L(\mu)$. Since every character appearing in a $\uq$ weight module is of the form (\ref{eq:characterwet}), this generalizes Proposition \ref{cor:herm2} to parameters for which finite dimensional $\uq$ weight modules are completely reducible as $\uqbs$ modules.
\end{rem}
\subsection{Invariance under the Wang-Zhang braid group operators}
In this section we show that each character of $\uqb$ that occurs in a $\uq$ weight module is invariant under the Wang-Zhang braid group operators.
We first prove this for the distinguished parameter, see (\ref{eq:distinguished}).
\begin{prop}\label{prop:charinv}
	Let $\boldsymbol{c}^\diamondsuit$ be the distinguished parameter. Then for each character $\chi\in \widehat{\uqdis}$ that occurs in a finite dimensional $\uq$ weight module and each $i\in \I_\circ $ it holds that
	$\chi\circ \boldsymbol{T}'_{i,-1}=\chi.$
\end{prop}
\begin{proof}
	Let $\lambda\in X^+$ be a weight such that  $V_\chi$ occurs in $L(\lambda)$, and suppose that $v\in L(\lambda)$ spans $V_\chi$. For each $i\in \I_\circ$ and $b\in \uqdis$ it holds that
	\begin{align*}
		b\big(\boldsymbol{T}'_{i,-1}\big)^{-1}v&=\big(\boldsymbol{T}'_{i,-1}\big)^{-1}(\boldsymbol{T}'_{i,-1}(b))\big(\boldsymbol{T}'_{i,-1}\big)^{-1}v\\
		&=\big(\boldsymbol{T}'_{i,-1}\big)^{-1}(\boldsymbol{T}'_{i,-1}(b)v)\\
		&=\big(\boldsymbol{T}'_{i,-1}\big)^{-1}(\chi\circ\boldsymbol{T}'_{i,-1}(b)v)\\
		&=\chi\circ\boldsymbol{T}'_{i,-1}(b)\big(\boldsymbol{T}'_{i,-1}\big)^{-1}(v).
	\end{align*}
	Therefore, the vector $\big(\boldsymbol{T}'_{i,-1}\big)^{-1}v$ spans a one-dimensional $\uqdis$ module where the action of $\uqdis$ is given by $\chi\circ \boldsymbol{T}'_{i,-1}$. By Lemma \ref{lem:tech3}, to show that $\chi\circ \boldsymbol{T}'_{i,-1}=\chi$, it suffices to show that
	$\chi\circ \boldsymbol{T}'_{i,-1}=\chi$
	on $\F[B_j\,:\, j\in \I_\circ]$. If $j\notin \I_{\mathrm{ns}}$, then $B_j \big(\boldsymbol{T}'_{i,-1}\big)^{-1}v=0=B_jv$, by Proposition \cite[\mbox{Prop 10.1. (i)}]{Kolb2023}. Hence it remains to show that the action of $B_j$ with $j\in \I_{\mathrm{ns}}$ coincides. This is carried out in Appendix \ref{Appendix} by a case-by-case analysis.
\end{proof}
In the following corollary, we show that the assumption on the parameter can be relaxed to balanced parameters.
\begin{cor}\label{cor:invariance}
	Let $\boldsymbol{d}$ be a balanced parameter. Then for each character $\chi\in \widehat{\uqd}$ that occurs in a finite dimensional $\uq$ weight module and each $w\in \Br(W^\Sigma) $ it holds that $\chi\circ \boldsymbol{T}'_{w,-1}=\chi.$
\end{cor}
\begin{proof}
	Let $\boldsymbol{c}^\diamondsuit$ be the distinguished parameter and define for each $i\in\I_\circ$ the scalar $a_i:=c_i^\diamondsuit d_i^{-1}$ and $a_i=1$ for $i\in\I_\bullet$. Recall $\Phi_{\boldsymbol{a}}$ from Section \ref{sec:iota}. This $\uq$ Hopf-algebra automorphism restricts to an algebra isomorphism $\uqd\to \uqdis$. Suppose that $\chi$ occurs in the $\uq$ module $L(\lambda)$. Then the character $\chi\circ \Phi_{\boldsymbol{a}}$ occurs in the twisted $\uq$ module $^{\Phi_{\boldsymbol{a}}}L(\lambda)\cong L(\lambda)$. Because $\Phi_{\boldsymbol{a}}\circ(\boldsymbol{T}'_{i,-1;\boldsymbol{c}^\diamondsuit})\circ\Phi_{\boldsymbol{a}}^{-1}=\boldsymbol{T}'_{i,-1;\boldsymbol{d}}$, see \cite[\mbox{Thm 9.10}]{Wang2023}, and by applying Proposition \ref{prop:charinv} to $\chi\circ\Phi_{\boldsymbol{a}}$ we have
	\begin{align*}
		\chi\circ \boldsymbol{T}'_{i,-1;\boldsymbol{d}}(b)&=\chi(\Phi_{\boldsymbol{a}}(\boldsymbol{T}'_{i,-1;\boldsymbol{c}}(\Phi_{\boldsymbol{a}}^{-1}(b)))\\
		&\stackrel{{Prop  \ref{prop:charinv}}}{=}\chi(\Phi_{\boldsymbol{a}}((\Phi_{\boldsymbol{a}}^{-1}(b)))\\
		&=\chi(b)\quad \text{for each}\quad b\in \uqd.
	\end{align*}
	As the reflections $\boldsymbol{r}_i$ for $i\in\I_\circ $ generate $\Br(W^\Sigma)$, we have $\chi\circ \boldsymbol{T}'_{w,-1}=\chi$ for each $w\in \Br(W^\Sigma)$.
\end{proof}
\subsection{Zonal spherical functions}
In this section we study zonal spherical functions, digressing from the overarching goal of studying spherical functions for general characters. We give a new proof of \cite[\mbox{Thm 5.3}]{Letzter2003} in rank one using elementary techniques. We start by showing an elementary property of zonal spherical functions. 
\begin{prop}\label{prop:sphgen} Let $\varphi$ be a zonal spherical function for $(\uqbs,\uqbs)$. Then $\varphi\circ S$ is a zonal spherical function for $(S^{-2}(\uqbs),\uqbs)$.  
\end{prop}
\begin{proof}
	 Recall that for each $x\in\uq$ we have $S^{-2}(x)=K_{2\rho}xK_{-2\rho}$, hence the algebra $S^{-2}(\uqbs)=\mathrm{ad}(K_{2\rho})(\uqbs)$ is a quantum symmetric pair coideal subalgebra. Also
	\begin{align*}
		\epsilon(x)=\epsilon(S(x))= S^{-1}(x_{(2)})x_{(1)}=x_{(1)}S(x_{(2)}).
	\end{align*} 
	Applying $S$ yields $S(x)K_{-2\rho}=K_{-2\rho}S^{-1}(x).$ 
	For all $b,b'\in \uqb$ and $x\in \uq$ it holds that
	{\allowdisplaybreaks
		\begin{align*}
			&\epsilon(S^{-2}(b'))\varphi(S(x))\epsilon(b)\\
			&=\epsilon(b)\varphi(S(x))\epsilon(S^{-1}(b'))\\
			&=\varphi\bigg( b_{(1)}S(b_{(2)})S(x) K_{2\rho}S(b'_{(2)})K_{-2\rho}b'_{(1)}\bigg)\\
			&=\varphi\bigg( S(\epsilon(b_{(1)})b_{(2)})S(x) K_{2\rho}S(\epsilon(b'_{(1)})b'_{(2)})K_{-2\rho}\bigg)\\
			&=\varphi\bigg( S(b)S(x) K_{2\rho}S(b')K_{-2\rho}\bigg)\\
			&=\varphi\bigg( S(b)S(x)S^{-1}(b')\bigg)\\
			&=\varphi\bigg( S(S^{-2}(b')xb)\bigg).
	\end{align*}}
	This shows that $\varphi\circ S$ is a zonal spherical function for $(S^{-2}(\uqbs),\uqbs)$. \qedhere
\end{proof}
Recall the notion of Weyl group invariance and affine Weyl group invariance for matrix entries, as introduced in Section \ref{subsec:invo}. Furthermore, recall the relative Weyl group from Theorem \ref{thm:dobsonkolb}. In rank one the relative Weyl group equals $W^\Sigma=\{\boldsymbol{r_1},1\}$. The action of $W^\Sigma$ on the quantum torus is given by $\boldsymbol{r_1}K_h=K_{-h}$, for $h\in Y_\Theta$. 
We combine Proposition \ref{prop:sphgen} and Lemma \ref{lem:tech4} to give a new proof of \cite[\mbox{Thm 5.3}]{Letzter2003} in rank one. In Theorem \ref{thm:weylgrpinv1}, we will show Weyl group invariance in general rank through a rank one reduction. Recall the quotient map $\cl:\A\to \C$ from Section \ref{sec:limit}.
\begin{thm}\label{thm:weylinvrank1}Let $(\I,\I_\bullet,\tau)$ be a Satake diagram of rank one. Then the following three statements hold:
	\begin{enumerate}[(i)]
		\item Each zonal spherical function $\varphi$ for $(\uqbs,\uqbs)$ satisfies $\varphi\circ S=\varphi\triangleleft K_{2\rho}$.
		\item Each zonal spherical function for $(\uqbs,\uqbs)$ is invariant under the affine Weyl group action.
		\item Each zonal spherical function for $(\mathrm{ad}(K_\rho)(\uqbs),\uqbs)$ is Weyl group invariant.
	\end{enumerate}
\end{thm}
\begin{proof}
	Let $\varphi=c_{f,v}\in L(\lambda)^\ast\otimes L(\lambda)$ be a zonal spherical function for $(\uqbs,\uqbs$). Remark that for Satake diagrams of rank one, the spherical weights are $-w_0$ invariant. Therefore, by Proposition \ref{prop:sphgen}, $\varphi\circ S\in L(\lambda)^\ast\otimes L(\lambda)$ is a zonal spherical function for $(S^{-2}(\uqbs),\uqbs)$, which is equal to the pair $(\text{ad}(K_{2\rho})(\uqbs),\uqbs)$. Furthermore, Remark \ref{lem:tech} implies that $c_{f ,v}\triangleleft K_{2\rho}$ is a zonal spherical function for $(\text{ad}(K_{2\rho})(\uqbs),\uqbs)$. 
	By Theorem \ref{thm:unique2} and Lemma \ref{lem:unique}, there exists a nonzero scalar $\eta\in\F$ such that \begin{equation}\label{eq:symetry}
	c_{f,v}\circ S=\eta c_{f ,v}\triangleleft K_{2\rho}.
	\end{equation}
Let $\mathcal B=\{v_{\mu_i}\,:\,1\leq i\leq m\}$ be a basis of $L(\lambda)$ consisting of weight vectors, and let $\mathcal B^\ast=\{v^\ast_{\mu_i}\,:\,1\leq i\leq m\}$ denote the dual basis of $L(\lambda)^\ast$. Without loss of generality we may assume that the vectors $f$ and $v$ are scaled such that
$v\in \text{span}_{\A}\mathcal B$ and $f\in \text{span}_{\A}\mathcal B^\ast $. Furthermore, by the proof of Proposition  \ref{prop:prop1} we can assume that $\cl (c_{f,v}(1))\neq 0$. Because $\cl (c_{f,v}(K_\rho))=\cl (c_{f,v}(1))\neq 0$, it follows that $c_{f,v}(K_\rho)\neq 0$. Therefore (\ref{eq:symetry}) implies that $\eta=1$.
This shows statement $(i).$ The relative Weyl group equals
	$W^\Sigma=\{1,\boldsymbol{r_1}\}$. The invariance is clear for the element $1\in W^\Sigma$. It remains to show the statement for $\boldsymbol{r_1}$. Let $h\in Y_{\Theta}$. Then we have $\boldsymbol{r_1}h=-h$. Applying the equality $c_{f,v}\circ S= c_{f ,v}\triangleleft K_{2\rho}$ yields
	$$
	\varphi(K_{h+\rho})=c_{f,v}(K_{h+\rho})
	=c_{f ,v}\triangleleft K_{2\rho}(K_{h-\rho})
	=c_{f,v}(K_{-h+\rho})
	=c_{f,v}(K_{\boldsymbol{r_1}h+\rho})
	=\varphi(K_{\boldsymbol{r_1}h+\rho}).$$
	This shows statement $(ii)$. Let $c_{g,w}\in L(\lambda)^\ast\otimes L(\lambda)$ be a zonal spherical function for the pair $(\mathrm{ad}(K_\rho)(\uqbs),\uqbs)$. By Remark \ref{lem:tech}, $c_{g ,w}\triangleleft K_{-\rho}$ is a zonal spherical function for $(\uqbs,\uqbs)$. Let $h\in Y_{\Theta}$. Using part $(ii)$, we have 
		$$
	c_{g,w}(K_h)=c_{g ,w}\triangleleft K_{-\rho}(K_{h+\rho})
	=c_{g ,w}\triangleleft K_{-\rho}(K_{\boldsymbol{r_1}h+\rho})
	=c_{g,w}(K_{\boldsymbol{r_1}h}).\qedhere
	$$
\end{proof}
\section{Bar involutions and relative Weyl group invariance}\label{sec:barqfa}
The goal of this section is to establish two kinds of symmetries for $\chi$-spherical functions. The first symmetry is in regard to $(\uqb,\uqb)$ spherical functions of type $\chi$. We show invariance under the braid group operators of Wang and Zhang \cite{Wang2023}. However, when restricted to the quantum torus the $(\uqb,\uqb)$ $\chi$-spherical functions do not lie in $\qfaaws$. The second symmetry establishes, for each quantum symmetric pair coideal subalgebra $\uqb$ a choice of coideal subalgebra $\uqds$, such that the $(\uqds,\uqb)$ spherical functions when restricted to the quantum torus lie in $\qfaaws$. One should compare these symmetries to the classical situation, where the $\chi$-spherical functions are invariant under the relative Weyl group action, cf. \cite[\mbox{Prop 5.1.2.}]{Heckman1995}.
\subsection{Symmetries for $(\mathbf{\uqb},\mathbf{\uqb})$ spherical functions}
The goal of this section is to show that $\chi$-spherical functions for $(\uqb,\uqb)$ are invariant under the Wang-Zhang braid group operators. We apply the invariance to study the behavior of $\chi$-spherical functions on the subalgebra $\uq^0=\F[K_h\,:\, h\in Y]$.
\begin{rem}\label{rem:nonzero}
	If $c_{f,v}\in L(\lambda)^\ast\otimes L(\lambda)$ is a nonzero $\chi$-spherical function for $(\uqb,\uqb)$, then by Proposition \ref{prop:prop1} it holds that $c_{f,v}(1)\neq 0$.
\end{rem}
\begin{prop}\label{prop:WZ}
	Let $\boldsymbol{c}$ be a balanced parameter, let $\chi\in \widehat{\uqb}$ be a character, and let  $\varphi\in{\qfachi}$ be a spherical function. Then for each $\varphi \in \qfachi$ and each $w\in \Br(W^\Sigma)$ it holds that $\varphi\circ \boldsymbol{T}_{w,-1}'=\varphi$.
\end{prop}
\begin{proof}
	It is enough to show that $\varphi\circ \boldsymbol{T}'_{i,-1}=\varphi$ for $\chi$-spherical functions $\varphi=c_{f,v}\in L(\lambda)^\ast\otimes L(\lambda)$ and $i\in \I_\circ$.
	We first show that $\varphi\circ \boldsymbol{T}'_{i,-1}\in L(\lambda)^\ast\otimes L(\lambda)$. 
	We have
	\begin{align*}
		c_{f,v}\circ \boldsymbol{T}'_{i,-1}(X)&=c_{  f\boldsymbol{T}'_{i,-1},(\boldsymbol{T}'_{i,-1})^{-1}v}(X),
	\end{align*}
	and $c_{ f\boldsymbol{T}'_{i,-1},(\boldsymbol{T}'_{i,-1})^{-1}(v)}\in L(\lambda)^\ast\otimes L(\lambda)$. The vector $  f\boldsymbol{T}'_{i,-1}$ spans a one-dimensional $\uqb$ module of type $\chi\circ(\boldsymbol{T}'_{i,-1})^{-1}$ and $(\boldsymbol{T}'_{i,-1})^{-1}v$ spans a $1$-dimensional $\uqb$ module of type $\chi \circ \boldsymbol{T}'_{i,-1}$. By Lemma \ref{lem:unique}, and Corollary \ref{cor:invariance}, this gives $ f(\boldsymbol{T}'_{i,-1})^{-1}=\eta_1f$ and  $ \boldsymbol{T}'_{i,-1}(v)=\eta_2v$ for nonzero scalars $\eta_1,\eta_2\in \F$. By Remark \ref{rem:nonzero} $\varphi(1)\neq0$, hence valuating at $1\in \uq$ gives us that 
	$$\varphi(1)=\varphi\circ \boldsymbol{T}'_{i,-1}(1)=c_{ f\boldsymbol{T}'_{i,-1},(\boldsymbol{T}'_{i,-1})^{-1}(v)}(1)=\eta_1\eta_2c_{f,v}(1),$$
	which implies that $\eta_1\eta_2=1$. As a consequence, it holds that $\varphi\circ \boldsymbol{T}'_{i,-1}=\varphi$.
\end{proof}
This gives us a certain braid group invariance on $\uq^0$. Since $\mathcal T'_{\boldsymbol{r_i},-1}(K_h)=K_{\boldsymbol{r_i}h}$, Proposition \ref{prop:WZ} and Theorem \ref{thm:braid} together imply that
\begin{equation}
	c_{f,v}(K_h)=c_{f\Upsilon_i,\Upsilon_i^{-1}v}(K_{\boldsymbol{r_i}h})\qquad \text{for all} \qquad h\in Y, c_{f,v}\in \qfachi,i\in \I_\circ.
\end{equation}
The longest element of the Weyl group $w_0\in W$ gives rise to a diagram automorphism $\tau_0$ that acts as on the simple roots by
\begin{equation}\label{eq:longestel}
	w_0\alpha_i=-\alpha_{\tau_0(i)}\qquad\text{for all}\qquad  i\in \I.
\end{equation}
The following proposition relates this diagram automorphism to the bar involution.
\begin{prop}\label{prop:invbar}
	Let $\boldsymbol{c}$ be a balanced uniform parameter, let $\chi\in\widehat{\uqb}$ a character,
	and let $c_{f,v}\in L(\lambda)^\ast\otimes L(\lambda)$ be a $\chi$-spherical function for $(\uqb,\uqb)$. Then there exists a nonzero scalar $a\in \F$ such that $\overline{ac_{f,v}(K_h)}=ac_{f,v}(K_{\tau_0(h)})$ for each $h\in Y$.
\end{prop}
\begin{proof}
	Let $c_{f,v}\in L(\lambda)^\ast\otimes L(\lambda)$ be a $\chi$-spherical function for $(\uqb,\uqb)$. By Lemma \ref{lem:upsion}, we may assume that $f$ and $v$ are invariant under the $\iota$-bar involutions. 
	The longest element $\boldsymbol{w}_0$ of $W^\Sigma$ equals $w_\bullet w_0$, cf. \cite{Dobson2019}. Let $\boldsymbol{w}_0=\boldsymbol{r}_{i_1}\dots \boldsymbol{r}_{i_n}$ be a reduced expression of $\boldsymbol{w}_0$, then the factorization of the quasi-$K$ matrix \cite[\mbox{Thm 8.1}]{Wang2023} implies that
	$\boldsymbol{T}'_{\boldsymbol{w}_0,-1}(K_h)=\Upsilon K_{\boldsymbol{w}_0h}\Upsilon^{-1}.$ As a result of Lemma \ref{lem:uqbullet}, we have $\chi|_{\uq_\bullet}=\epsilon$. Let $h\in Y$.
	By applying Proposition \ref{prop:WZ} to $\boldsymbol{w}_0$ and using that $\chi|_{\uq_\bullet}=\epsilon$, it follows that have\allowdisplaybreaks{
		\begin{align*}
			c_{f,v}(K_h)&\stackrel{{Prop \,\ref{prop:WZ}}}{=}c_{f \Upsilon,\Upsilon^{-1}v}(K_{w_\bullet w_0h})\\
			&\stackrel{{(\ref{eq:longestel})}}{=}c_{f \Upsilon,\Upsilon^{-1}v}(K_{-w_\bullet \tau_0(h)})\\
			&\stackrel{{Lem\, \ref{lem:upsion}}}{=}c_{\overline{f},\overline{v}}(K_{-w_\bullet \tau_0(h)})\\
			&=c_{\overline{f},\overline{v}}(K_{-\tau_0(h)+\mu})\qquad \text{where}\qquad K_\mu\in \uq_\bullet\\
			&\stackrel{{Lem\, \ref{lem:uqbullet}}}{=}c_{\overline{f},\overline{v}}(K_{-\tau_0(h)})\\
			&\stackrel{{ (\ref{eq:barinv})}}{=}\overline{c_{f,v}(K_{\tau_0(h)})}.
	\end{align*}}
	This is the desired result.
\end{proof}
\subsection{Symmetries for $(\mathbf{\uqb^\chi},\mathbf{\uqb})$ spherical functions}\label{sec:weylinv}
In this section we turn our attention to relative Weyl group symmetries when restricted to the quantum torus. We show that for each coideal subalgebra $\uqb$, there exists a coideal algebra $\uqds$ such that the related spherical functions are Weyl group invariant. The strategy for showing Weyl group invariance of spherical functions in rank one is as follows: let $v\in L(\lambda)$ be a $\chi$-spherical vector for $\uqb$. Through the Shapovalov form one may identify $\overline{v}\in L^r(\lambda)\cong L(\lambda)^\ast$ as a linear functional on $L(\lambda)$. Using the Wang-Zhang braid group operators one shows that $\Res(c_{\overline{v},v})\in \qfaaws$. Hence, we want to find a coideal subalgebra $\uqds$ such that $\overline{v}\in L^r(\lambda)$ is a spherical vector. We establish this in Proposition \ref{prop:finaly}. We start by stating a preliminary technical lemma regarding the action of $T_{w_\bullet,-1}''(E_i)$ on spherical vectors. 
\begin{lem}\label{lem:lusztiglong}
	Let $\lambda\in X^+$ be a dominant weight, and let $v\in L(\lambda)$ be a $\chi$-spherical vector for $\uqbs$. Then for each $i\in \I_\circ$ it holds that $T_{w_\bullet,-1}''(E_i)v=T_{w_\bullet,+1}''(E_i)v$.
\end{lem}
The proof or Lemma \ref{lem:lusztiglong} will be given in Appendix \ref{apenB}. Let $v\in L(\lambda)$ be a spherical vector. With help of Lemma \ref{lem:lusztiglong}, we determine a coideal subalgebra $\uqds$ such that $\overline{v}\in L^r(\lambda)$ is a spherical vector for $\uqds$. 
\begin{prop}\label{prop:finaly}
	Let $\boldsymbol{c}$ be a balanced parameter with $\overline{c_i}=c_i^{-1}$, and let $v\in L(\lambda)$ be a $\chi$-spherical vector for $\uqb$. Then $\overline{v}\in L^r(\lambda)$ is a  $\chi'$-spherical vector for $\uqds$, where
	\begin{align}\label{eq:parametershift}
		d_i&=c_i\chi(K_i K_{\tau(i)}^{-1})(-1)^{\langle 2\rho_{\bullet}^\vee,\alpha_i\rangle}
		q_i^{-\langle \alpha_i^\vee,2\rho_\bullet\rangle}q_i^{\langle\alpha_i^\vee,\alpha_i\rangle}&&\text{and}\\
		t_i&=-c_i\overline{\chi(B_i^{\boldsymbol{c}})}q_i^{2},&&\text{for}\qquad i\in\I_\circ,\nonumber
	\end{align}
	and the character $\chi'\in\widehat{\uqd}$ satisfies
	\begin{align}\label{eq:character}
		\chi'|_{\uq_\bullet \uq_\Theta^0}=\chi|_{\uq_\bullet \uq_\Theta^0}\quad\text{ and}\quad
		\chi'(B_i)=0\quad\text{for }i\in\I_\circ.
	\end{align}
\end{prop}
\begin{proof}
	By Lemma \ref{lem:uqbullet}, the action of $ \uq_\Theta^0$ is determined by the dominant weight $\lambda$ and the action of $\uq_\bullet$ on $\overline{v}\in L^r(\lambda)$ is equal to the action of the counit. It remains to determine the action of $B_i$ for $i\in \I_\circ.$ 
	Using (\ref{eq:lusztigrel}) we deduce that
	\begin{align}\label{eq:rhorel}
		\varrho (T_{w_\bullet,+1}''(E_{\tau(i)}))&=\sigma\circ \sigma \circ \varrho (T_{w_\bullet,+1}''(E_{\tau(i)}))\\
		&=\sigma\circ \omega (T_{w_\bullet,+1}''(E_{\tau(i)}))\nonumber\\
		&=\sigma (T_{w_\bullet,+1}'(F_{\tau(i)}))\nonumber\\
		&=T_{w_\bullet,-1}''(F_{\tau(i)}),\nonumber
	\end{align}
	and $\overline{T''_{w_\bullet,-1}(F_{\tau(i)})}=T''_{w_\bullet,+1}(F_{\tau(i)})$.
	By Lemma \ref{lem:lusztiglong} and because $T_{w_\bullet,+1}''(E_{\tau(i)})$ $q$-commutes with $K_i^{-1}$, we have 
	\begin{equation}\label{eq:twonmodule}
		T''_{w_\bullet,-1}(E_i)K_{\tau(i)}^{-1}v=T''_{w_\bullet,+1}(E_i)K_{\tau(i)}^{-1}v,
	\end{equation}
	and by
	\cite[\mbox{Lem 2.9}]{Balagovic2015} the following equality holds \begin{equation}\label{eq:BK}
		T''_{w_\bullet,-1}(E_i)=
		(-1)^{\langle 2\rho_{\bullet}^\vee,\alpha_i\rangle}
		q_i^{-\langle \alpha_i^\vee,2\rho_\bullet\rangle }T'_{w_\bullet,-1}(E_i).
	\end{equation}
	Using these relations, we compute for $i\notin \I_{\mathrm{ns}}$ that \allowdisplaybreaks{
		\begin{align*}
			&\overline{\varrho(B^{\boldsymbol{d,t}}_i)}v \stackrel{\mathclap{(\ref{eq:rhorel})}}{=} \overline{(E_i+d_i K_i^{-1}T_{w_\bullet,-1}''(F_{\tau(i)}))}v\\
			&\stackrel{\mathclap{(\ref{eq:lusztigrel})}}{=} \big(E_i+d_i^{-1} K_iT_{w_\bullet,+1}''(F_{\tau(i)})\big)v\\
			&= \big(E_i+d_i^{-1} K_iT_{w_\bullet,+1}'' F_{\tau(i)}T_{w_\bullet,-1}'\big)v\\
			&\stackrel{{Lem\, \ref{lem:uqbullet}}}{=} \big(E_i+d_i^{-1} K_iT_{w_\bullet,+1}'' F_{\tau(i)}\big)v\\
			&= d_i^{-1}K_iT_{w_\bullet,+1}''\big(d_i T_{w_\bullet,-1}'K_i^{-1}E_i+ F_{\tau(i)}\big)v\\
			&= d_i^{-1}K_iT_{w_\bullet,+1}''\big(d_iq_i^{-\langle\alpha_i^\vee,\alpha_i
				\rangle}  T_{w_\bullet,-1}'E_iK_i^{-1}+ F_{\tau(i)}\big)v\\
			&\stackrel{{Lem\, \ref{lem:uqbullet}}}{=} d_i^{-1}K_iT_{w_\bullet,+1}''\big(d_iq_i^{-\langle\alpha_i^\vee,\alpha_i
				\rangle}  T_{w_\bullet,-1}'E_iK_i^{-1}T_{w_\bullet,+1}''+ F_{\tau(i)}\big)v\\
			&= d_i^{-1}K_iT_{w_\bullet,+1}''\big(d_iq_i^{-\langle\alpha_i^\vee,\alpha_i
				\rangle}  T_{w_\bullet,-1}'E_iT_{w_\bullet,+1}''K_{w_\bullet d_i\alpha_i^\vee}^{-1}+ F_{\tau(i)}\big)v\\
			&\stackrel{{Lem\, \ref{lem:uqbullet}}}{=}d_i^{-1}K_iT_{w_\bullet,+1}''( d_iq_i^{-\langle\alpha_i^\vee,\alpha_i
				\rangle} T_{w_\bullet,-1}'(E_i)K_i^{-1}+ F_{\tau(i)})v\\
			&\stackrel{\mathclap{(\ref{eq:BK})}}{=}d_i^{-1}K_iT_{w_\bullet,+1}''(d_i(-1)^{\langle 2\rho_{\bullet}^\vee,\alpha_i\rangle}q_i^{\langle \alpha_i^\vee,2\rho_\bullet\rangle } q_i^{-\langle\alpha_i^\vee,\alpha_i
				\rangle} T_{w_\bullet,-1}''(E_i)K_i^{-1}+ F_{\tau(i)})v\\
			&\stackrel{{(\ref{eq:twonmodule})}}{=}d_i^{-1}K_iT_{w_\bullet,+1}''(\chi(K_i^{-1}K_{\tau(i)})d_i(-1)^{\langle 2\rho_{\bullet}^\vee,\alpha_i\rangle}q_i^{\langle \alpha_i^\vee,2\rho_\bullet\rangle }q_i^{-\langle\alpha_i^\vee,\alpha_i
				\rangle} T_{w_\bullet,+1}''(E_i)K_{\tau(i)}^{-1}+ F_{\tau(i)})v.
	\end{align*}}
	By definition of $d_i$ we have $\chi(K_i^{-1}K_{\tau(i)})d_i(-1)^{\langle 2\rho_{\bullet}^\vee,\alpha_i\rangle}
	q_i^{\langle \alpha_i^\vee,2\rho_\bullet\rangle}q_i^{-\langle\alpha_i^\vee,\alpha_i
		\rangle} =c_i$. By \cite[\mbox{Prop 10.1. (i) }]{Kolb2023} it follows that
	$$\underbrace{\big(\chi(K_i^{-1}K_{\tau(i)})d_i(-1)^{\langle 2\rho_{\bullet}^\vee,\alpha_i\rangle}
		q_i^{\langle \alpha_i^\vee,2\rho_\bullet\rangle}q_i^{-\langle\alpha_i^\vee,\alpha_i
			\rangle}T''_{w_\bullet,+1}(E_i)K_{\tau(i)}^{-1}+ F_{\tau(i)}\big)}_{=B_{\tau(i)}^{\boldsymbol{c}}\in\uqb}v=0$$
	and hence $\varrho(B^{\boldsymbol{d,t}}_i)\overline{v}=\overline{\overline{\varrho(B^{\boldsymbol{d,t}}_i)}v}=0$. 
	Next, let $i\in\I_{\mathrm{ns}}$. Using that $\overline{\varrho(F_i+c_iE_iK_i^{-1})}=c_i^{-1}q_i^{-2}K_i(F_i+c_iE_iK_i^{-1})$, we have
	\begin{align*}
		\overline{\varrho(B_i^{\boldsymbol{d,t}})}v&=\overline{\varrho(B_i^{\boldsymbol{c}})-\overline{\chi(B_i^{\boldsymbol{c}})}c_iq_i^{2}K_i^{-1}}v\\
		&=(c_i^{-1}q_i^{-2}K_i\underbrace{(F_i+c_iE_iK_i^{-1})}_{=B_i^{\boldsymbol{c}}\in \uqb}-c_i^{-1}q_i^{-2}\chi(B_i^{\boldsymbol{c}})K_i)v\\
		&=(c_i^{-1}q_i^{-2}\chi(B_i^{\boldsymbol{c}})K_i-c_i^{-1}q_i^{-2}\chi(B^{\boldsymbol{c}}_i)K_i)v\\
		&=0,
	\end{align*}
	and hence $\varrho(B_i^{\boldsymbol{d,t}})\overline{v}=\overline{\overline{\varrho(B^{\boldsymbol{d,t}}_i)}v}=0$. 
	This shows that $\overline{v}\in L^r(\lambda)$ is a spherical vector for $\uqds$ satisfying (\ref{eq:character}).
\end{proof}
\begin{rem}
	Note that the parameter $\boldsymbol{d}$ from Proposition \ref{prop:finaly} satisfies $\overline{d_i}=d_i^{-1}$, for each $i\in\I_\circ.$
\end{rem}
Recall that a parameter $\boldsymbol{c}$ is called admissible if $\boldsymbol{c}$ is balanced, satisfies $\overline{c_i}=c_i^{-1}$ and $c_{i}=q^{-(\alpha_i,w_\bullet \alpha_{\tau(i)}+2\rho_{\bullet})} \overline{c_{\tau(i)}}$ for each $i\in \I_{\circ}$, cf. (\ref{eq:admissible}).
Let $\boldsymbol{c}$ be an admissible parameter and let $\chi\in\widehat{\uqb}$ be a character that occurs in a finite dimensional $\uq$ weight module. Then we use $\uqb^\chi:=\uqds$ to denote the coideal subalgebra defined in Proposition \ref{prop:finaly}. Note that Proposition \ref{prop:finaly} depends on a dominant weight $\lambda\in X^+$, however $\uqb^\chi$ is independent of the weight $\lambda$. We extend the definition of $\mathbf{B}_{\mathbf{c}}^\chi$ to general balanced parameters $\boldsymbol{r}$ as follows: let $\boldsymbol{r}$ be a balanced parameter, let 
$\chi\in\widehat{\mathbf{B}_{\mathbf{r}}}$
be a character that occurs in a finite dimensional $\uq$ weight module, and let $\boldsymbol{a}\in (\F^\times)^{\I}$ with $a_j=1$ for $j\in\I_\bullet$ be such that $\Phi_{\boldsymbol{a}}:\uqb\to \mathbf{B}_{\mathbf{r}}$ is an isomorphism. If $V_\chi$ occurs in $L(\lambda)$, then $\chi\circ \Phi_{\boldsymbol{a}}$ is a character of $\uqb$ that occurs in $^{\Phi_{\boldsymbol{a}}}L(\lambda)\cong L(\lambda)$.
Henceforth, using Proposition \ref{prop:finaly} we may set $\mathbf{B}_{\mathbf{r}}^\chi:=\Phi_{\boldsymbol{a}}\big(\mathbf{B}_{\mathbf{c}}^{\chi\circ \Phi_{\boldsymbol{a}}}\big)$. 
\begin{rem}
We remark that that the definition of $\mathbf{B}_{\mathbf{r}}^\chi$ seems to depend on the choice of admissible parameter $\boldsymbol{c}$. By Lemma \ref{lem:muller} there exists for each $i\in\I_\circ$ an integer $l_i\in\Z$ with $\chi\circ \Phi_{\boldsymbol{a}}(B^{\mathbf{c}}_i)=[l_i]_{q_i}\sqrt{c_iq_i}$. Hence, the parameter $(\boldsymbol{d},\boldsymbol{t})$ of the coideal subalgebra $\mathbf{B}_{\mathbf{r}}^\chi$ is given by
\begin{align}\label{eq:parametershift2}
	d_i&=r_i\chi(K_i K_{\tau(i)}^{-1})(-1)^{\langle 2\rho_{\bullet}^\vee,\alpha_i\rangle}
	q_i^{-\langle \alpha_i^\vee,2\rho_\bullet\rangle}q_i^{\langle\alpha_i^\vee,\alpha_i\rangle}&&\text{and}\\
	t_i&=-r_i^{1/2}c_i^{-1/2}\cdot c_i \overline{\chi\circ \Phi_{\boldsymbol{a}}(B^{\mathbf{c}}_i)}q_i^{2},&&\text{for}\qquad i\in\I_\circ.\nonumber
\end{align}
For each $i\in\I_\circ$ the scalars $t_i=
-r_i^{1/2}c_i^{-1/2} c_i\overline{c_i^{1/2}}
 \overline{[l_i]_{q_i}\sqrt{q_i}}q_i^{2}
 =-r_i^{1/2} \overline{[l_i]_{q_i}\sqrt{q_i}}q_i^{2}$ and $d_i$ do not depend on the choice of admissible parameter $\boldsymbol{c}$.
\end{rem}
Let 
\begin{align*}
B_i^{\boldsymbol{r},\chi}&= F_i+r_i\chi(K_i K_{\tau(i)}^{-1})(-1)^{\langle 2\rho_{\bullet}^\vee,\alpha_i\rangle}
q_i^{-\langle \alpha_i^\vee,2\rho_\bullet\rangle}q_i^{\langle\alpha_i^\vee,\alpha_i\rangle} T''_{w_\bullet,+1}(E_{\tau(i)})K_i^{-1}\\
&\quad-r_i^{1/2} \overline{[l_i]_{q_i}\sqrt{q_i}}q_i^{2}K_i^{-1}
\end{align*}
with $i\in\I_\circ$ denote generators of $\mathbf{B}_{\mathbf{r}}^\chi$.
\begin{defi}[Akin characters]
	Let $\boldsymbol{c}$ be a balanced parameter, and let $\chi$ and $\chi'$ be characters of ${\uqb}$ and ${\uqbschi}$ respectively. Then we say that the character $\chi'$ is \textit{akin} to the character $\chi$ if the following two conditions hold:
	\begin{enumerate}[(i)]\label{enum:eq}
		\item For each $x\in\uq_\bullet \uq_\Theta^0$ it holds that $\chi'(x)=\chi(x)$.
		\item For each $i\in\I_\circ$ it holds that $\chi'(B_i^{\boldsymbol{c},\chi})=0$, with $B_i^{\boldsymbol{c},\chi}$ defined above.\\
	\end{enumerate}
	We write $\qfachb$ to denote the $\F$-vector space of $(\chi',\chi)$-spherical functions for $(\uqbschi,\uqb)$, where $\chi'$ is akin to $\chi$.
\end{defi}
Note that the character $\chi'\in\widehat{\uqb^\chi}$ is uniquely determined by the equations given in Definition \ref{enum:eq}.
\begin{rem}
	Let $\boldsymbol{c}$ be an admissible parameter. By Proposition \ref{prop:finaly} the space of spherical functions $\qfachb$ is nontrivial. Using the isomorphism $\Phi_{a}$ and Lemma \ref{lem:tech4}, it follows that the space is nontrivial for any balanced parameter.
\end{rem}
We continue to show that in rank one, the spherical functions in $\qfachb$ are Weyl group invariant. For this we make use of the Shapovalov form as introduced in (\ref{eq:Shapovalof}). We first need a preliminary Lemma regarding the evaluation of the Shapovalov form. Recall the quotient maps $\cl$ from Section \ref{sec:limit}. Fix an integer $d\geq 1$ and let $\uq_{\Q(q^{1/d})}$ denote the $\Q(q^{1/d})$ algebra generated by the elements $E_i,F_i$ and $K_i$ for $i\in \I$. 
Similarly, let $\mathbf{U}_\Q(\mathfrak{g})$ denote the $\Q$ algebra generated by the elements $\cl(E_i),\cl(F_i)$ and $ \cl\big(\frac{K_i-1}{q-1}\big)$ for $i\in \I$. We write $\varkappa: \mathbf{U}(\mathfrak{g})\to \mathbf{U}(\mathfrak{g})$ for the algebra anti-automorphism that satisfies
$$\cl (E_i)\mapsto \cl (F_i),\quad \cl (F_i)\mapsto \cl(E_i),\quad \cl\big(\frac{K_i-1}{q-1}\big)\mapsto \cl\big(\frac{K_i-1}{q-1}\big),\quad \text{for}\quad  i\in \I.$$
According to \cite[\mbox{II 2.3.2}]{Kumar2002} the $\mathbf{U}_\Q(\mathfrak{g})$ module $L_\Q(\lambda)^1:=\mathbf{U}_\Q(\mathfrak{g}) \cl(v_\lambda)$ admits a unique symmetric, non-degenerate bilinear form $S^\lambda$ with $S^\lambda(\cl (v_\lambda),\cl (v_\lambda))=1$ satisfying
\begin{equation}\label{eq:contravariant}
S^\lambda( X v,w)=S^\lambda(v, \varkappa(X)w),  \qquad\text{where}\qquad v,w\in L(\lambda)^1,\, X\in \mathbf{U}(\mathfrak{g}).
\end{equation}
Bilinear forms satisfying (\ref{eq:contravariant}) are referred to as \textit{contravariant}. Additionally, for each dominant weight $\lambda\in X^+$ we set $L_{\Q(q^{1/d})}(\lambda):= \uq_{\Q(q^{1/d})} v_\lambda$. 
\begin{lem}\label{eq:notzero}
Let $\lambda\in X^+$ be a dominant weight and $d\geq 1$ be a integer. Then for each $0\neq v\in L_{\Q(q^{1/d})}(\lambda)$ we have $(\overline{v},v)\neq 0$.
\end{lem}
\begin{proof}
Without loss of generality, we may assume that $v\in L^\A(\lambda)$ with $\cl(v)\neq 0$. Let us define a bilinear form $S^\lambda:L(\lambda)^1\times L(\lambda)^1\to \C$ by setting
$$S^\lambda(\cl (v),\cl (w)):=\cl \big((v,w)\big),\qquad \qquad \text{for}\qquad v,w\in L^\A(\lambda).$$
It is readily seen that $S^\lambda$ is a well-defined, contravariant, non-degenerate symmetric bilinear form with $S^\lambda(\cl (v_\lambda),\cl (v_\lambda))=1$. As a consequence of \cite[\mbox{II 2.3.2 \& 2.3.13}]{Kumar2002} the bilinear form $S^\lambda$ is positive definite on $L_\Q(\lambda)^1$. Furthermore, since
$\cl (\widehat{\uq}\cap \uq_{\Q(q^{1/d})})=\mathbf{U}_\Q(\mathfrak{g})$ and $\cl(\overline{v})=\cl (v)$ we deduce that
\begin{align*}
0\neq S^\lambda(\cl (v),\cl (v))=S^\lambda(\cl (\overline{v}),\cl (v))= \cl (\overline{v},v).
\end{align*}
Thus, it follows that $(\overline{v},v)\neq 0$.
\end{proof}
Let $\boldsymbol{s}\in (\F^\times)^{\I}$ denote the tuple given by $s_i:=(-1)^{\langle 2\rho^\vee_\bullet,\alpha_i\rangle}$, for $i\in \I$.
\begin{prop}\label{prop:rang1}
	Let $(\I,\I_\bullet,\tau)$ be a Satake diagram of rank one and let $\boldsymbol{c}$ be an admissible parameter. Then each spherical function $\varphi\in {\qfachb}$ satisfies 
	$$\varphi(x)=\varphi\big(\Phi_{\boldsymbol{s}}\circ\varrho\circ \mathcal T_{\boldsymbol{r_i},-1}'(x)\big),\qquad \text{for each}\qquad x\in\uq,i\in\I_\circ.$$
	In particular,
	$\varphi$ is Weyl group invariant.
\end{prop}
\begin{proof}
	First suppose that the parameter $\boldsymbol{c}$ is uniform. Let $\varsigma\in {\qfachb}\cap L^\ast(\lambda)\otimes L(\lambda)$ be a nonzero spherical function. By Theorem \ref{thm:unique2} there exists a $\chi$-spherical vector $v\in L(\lambda)$ for $\uqb$, which we may assume to be $\iota$-bar invariant by Lemma \ref{lem:upsion}. Using Proposition \ref{prop:finaly}, we see that the matrix entry 
	$$\big[x\mapsto \varphi(x):=(\overline{v},xv)\big]\in {\qfachb}$$ is a spherical function for $(\uqb^\chi,\uqb)$.
	By Theorem \ref{thm:unique2}, it follows that the space ${\qfachb}\cap L(\lambda)^\ast\otimes L(\lambda)$ is spanned by $\varphi$. So it suffices to show that $\varphi$ is Weyl group invariant. 
	By Proposition \ref{prop:charinv} and Lemma \ref{lem:unique}, we have
	\begin{align*}
		(\boldsymbol{T}_{i,-1}')^{-1}v&=(\mathcal T_{\boldsymbol{r_i},-1}')^{-1}\Upsilon^{-1}v=\vartheta(i)v\qquad\text{where}\quad 0\neq \vartheta(i)\in \F.
	\end{align*}
	Consequently it holds that \begin{equation}\label{eq:important}
	(\overline{(\boldsymbol{T}'_{i,-1})^{-1}v},x(\boldsymbol{T}'_{i,-1})^{-1}v)=\overline{\vartheta(i)}\vartheta(i)(\overline{v},xv),\qquad \text{for each }x\in \uq
	\end{equation}
	Let $x\in \uq$.
	Using that $\overline{\varrho((\mathcal{T}'_{\mathbf{r}_i,-1})^{-1})}=\mathcal{T}'_{\mathbf{r}_i,-1}$, cf. \cite[\mbox{8.2 \& 8.4}]{Jantzen1996}, we have \begin{align}\label{eq:varrho}
		(\overline{(\boldsymbol{T}'_{i,-1})^{-1}v},x(\boldsymbol{T}'_{i,-1})^{-1}v)&=(v, \mathcal T_{\boldsymbol{r_i},-1}'x(\mathcal T_{\boldsymbol{r_i},-1}')^{-1}\overline{v})\\
		&=(v,  \mathcal T_{\boldsymbol{r_i},-1}'(x) \overline{v})\nonumber\\
		&=(\mathcal \varrho(\mathcal T_{\boldsymbol{r_i},-1}'(x))v,  \overline{v}\nonumber)\\
		&=(\overline{v},\varrho(\mathcal T_{\boldsymbol{r_i},-1}'(x))v)\nonumber.
	\end{align}
	It remains to show that $\vartheta(i)\overline{\vartheta(i)}=1$. Using (\ref{eq:important}) and $(\ref{eq:varrho})$ we see that it is sufficient to show that $(\overline{v},v)\neq 0$. As $\boldsymbol{c}$ is admissible there exists an integer $d\geq1$ such that $c_i\in \Q(q^{1/d})$ for each $i\in\I_\circ$. Therefore, without loss of generality we may assume that $v\in  L_{\Q(q^{1/d})}(\lambda)$. Using Lemma \ref{eq:notzero} we conclude that $(\overline{v},v)\neq 0$.
In the case where the admissible parameter is not uniform, a similar argument remains valid. Recall that $\boldsymbol{s}\in (\F^\times)^{\I}$ denotes the tuple given by $s_i:=(-1)^{\langle 2\rho^\vee_\bullet,\alpha_i\rangle}$, for $i\in \I$. One may invoke the intertwining properties of the quasi-$K$ matrix to show that $v$ may be scaled such that $\Upsilon ^{-1} v= \overline{\Phi_{\boldsymbol{s}}(v)}=\Phi_{\boldsymbol{s}}(\overline{v})$. By applying Lemma \ref{lem:tech4} it follows that{\belowdisplayskip=-12pt
\begin{align*}
(\overline{v},xv)&=(\overline{(\boldsymbol{T}'_{i,-1})^{-1}v},x(\boldsymbol{T}'_{i,-1})^{-1}v)\\
&=(\Phi_{\boldsymbol{s}}(\overline{v}),\varrho(\mathcal T_{\boldsymbol{r_i},-1}'(x))\Phi_{\boldsymbol{s}}(v))\\
&=(\overline{v},\Phi_{\boldsymbol{s}}\circ \varrho\circ\mathcal T_{\boldsymbol{r_i},-1}'(x)v)\nonumber.
\end{align*}\qedhere}
\end{proof}
Through a rank one reduction we show Weyl group invariance for general rank. Recall the space $Y_{\Theta}=\{h\in Y:\Theta(h)=-h\}$ from (\ref{eq:thethainv}) and the quantum torus $\mathcal A=\F[K_h:h\in Y_\Theta]$.
\begin{thm}\label{thm:weylgrpinv1}
	Let $\boldsymbol{c}$ be an admissible parameter. Then restriction to the quantum torus defines a $\F$-linear map
	$$\mathrm{Res}: \bigoplus_{\chi\in \widehat{\uqb}}{^{\uqb^\chi} \qfa^{\uqb}_\chi}\to \qfaaws.$$
\end{thm}
\begin{proof}
	Let $\varphi=c_{f,v}\in L^\ast(\lambda)\otimes L(\lambda)$ be a $(\chi',\chi)$-spherical function for $(\uqb^\chi,\uqb)$, where $\chi'$ is akin to $\chi$.
	Let $i\in \I_\circ$ and set $Y_{i,\Theta}^\perp=\{h\in Y_{\Theta}: \langle h,\alpha_i\rangle=0\}$. The free $\Z$-module $Y_{\Theta}$ has a decomposition $Y_{\Theta}=Y_{i,\Theta}^\perp\oplus Y_{i,\Theta}$. 
	Recall the rank one coideal subalgebra $\uqb^i\subset \uq_i$ corresponding to the rank one diagram Satake diagram $(\{i,\tau(i)\}\cup \I_\bullet,\I_\bullet,\tau)$ as introduced in Section \ref{subsec:qsp}. Let $h\in  Y_{i,\Theta}^\perp$. 
	Consider the matrix coefficient
	$$\widehat{\varphi}: \uq_i\to \F\qquad\text{where}\qquad x\mapsto \widehat{\varphi}(x):=\varphi (XK_{h})\qquad\text{for}\qquad x\in \uq_i.$$
	We claim that $\widehat{\varphi}$ is a spherical function for $((\uqb^i)^{\chi},\uqb^i)$. 
	Because $$\langle h,\Theta(\alpha_i)\rangle=\langle \Theta(h),\alpha_i\rangle=\langle -h,\alpha_i\rangle=0,$$
	it holds that 
	\begin{align*}
		K_h T''_{w_\bullet,+1}(E_{\tau(i)})
		&=q^{\langle h,\alpha_i\rangle}T''_{w_\bullet,+1}(E_{\tau(i)})K_h\\
		&=T''_{w_\bullet,+1}(E_{\tau(i)})K_h,
	\end{align*}
	and 
	\begin{align*}
		K_h F_i=q^{-\langle h,\alpha_i\rangle}F_iK_h=F_iK_h.
	\end{align*}
	This means that $B_iK_h=K_hB_i$.
	A similar argument shows that $B_{\tau(i)} K_h=B_{\tau(i)} K_h$. 
	Trivially $K_h$ commutes with $\uq_{\Theta}^0$. 
	Let $j\in \I_\bullet$, then it holds that $$K_hE_j= q^{\langle h,\alpha_j\rangle }E_jK_h\qquad\qquad K_hF_j= q^{-\langle h,\alpha_j\rangle}F_jK_h.$$
	Since $\chi|_{\uq_\bullet}=\epsilon$, we may conclude that $\widehat{\varphi}$ is a $(\chi',\chi)$-spherical function for the pair
	$((\uqb^\chi)^i,\uqb^i)=((\uqb^i)^{\chi},\uqb^i)$. Because $\chi'$ is akin to $\chi$, the restrictions remain akin.
	Let $a\in Y_{i,\Theta}\cap Y_{i}$, where $Y_i\subset Y$ is part of the root datum corresponding to the quantum group $\uq_i$. 
	By Theorem \ref{prop:rang1}, we have $\widehat{\varphi}(K_{\boldsymbol{r_i} a})=\widehat{\varphi}(K_{a})$. Furthermore, using that $\langle h,\alpha_i\rangle=0$ and $\langle h,\Theta(\alpha_i)\rangle=0$ we have $\boldsymbol{r_i} h=h$. Using these relations we compute that
	\begin{align*}
		\varphi(K_{h+a})&=\widehat{\varphi}(K_{a})\\
		&=\widehat{\varphi}(K_{\boldsymbol{r_i} a})\\
		&=\varphi(K_{h+\boldsymbol{r_i} a})\\
		&=\varphi(K_{\boldsymbol{r_i}h+\boldsymbol{r_i} a})\\
		&=\varphi(K_{\boldsymbol{r_i}(h+ a)}).
	\end{align*}
	As $Y_{i,\Theta}^\perp\oplus (Y_{i,\Theta}\cap Y_{i})=Y_{\Theta}$, this shows that $\varphi$ is Weyl group invariant.
\end{proof}
The remaining goal is to relax the assumption on the parameter. 
\begin{thm}\label{thm:mainthm}
	Let $\boldsymbol{d}$ be a balanced parameter. Then restriction to the quantum torus defines a $\F$-linear map
	$$\mathrm{Res}: \bigoplus_{\chi\in \widehat{\uqd}}{^{\uqd^\chi} \qfa^{\uqd}_\chi}\to \qfaaws.$$
\end{thm}
\begin{proof}
	Let $\boldsymbol{c}$ be an admissible parameter and set for $i\in\I_\circ$ the scalar $a_i:=c_id_i^{-1}$ and $a_i=1$ for $i\in \I_\bullet$. Then $\Phi_{\boldsymbol{a}}$ restricts to an algebra isomorphism $\uqd\to \uqb$ and $\uqd^\chi\to \uqb^{\chi\circ \Phi_{\boldsymbol{a}}^{-1}}$. Let $c^{L(\lambda)}_{f,v}\in{^{\uqd^\chi} \qfa^{\uqd}_\chi}$ be a spherical function. Lemma \ref{lem:tech4} implies that $c^{L(\lambda)}_{\Phi^\ast_{\boldsymbol{a}}(f),\Phi_{\boldsymbol{a}}(v)}$ is a $\chi\circ \Phi_{\boldsymbol{a}}^{-1}$-spherical function for $(\uqb^{\chi\circ \Phi_{\boldsymbol{a}}^{-1}},\uqb)$ in $^{\uqb^{\chi\circ \Phi_{\boldsymbol{a}}^{-1}}}\qfa^{\uqb}_{\chi\circ \Phi_{\boldsymbol{a}}^{-1}}$. 
	According to Theorem \ref{thm:weylgrpinv1} and Lemma \ref{lem:transformation}, we have
	$$\Res(c^{L(\lambda)}_{f,v})\stackrel{Lem\, \ref{lem:transformation}}{=}\Res\bigg(c^{L(\lambda)}_{\Phi^\ast_{\boldsymbol{a}}(f),\Phi_{\boldsymbol{a}}(v)}\bigg)\stackrel{Thm \,\ref{thm:weylgrpinv1}}{\in} \qfaaws .\qedhere$$
\end{proof}
\begin{rem}
	Weyl group invariance for spherical functions for quantum symmetric pairs over $\Q(q)$, or a sub-field of $\F$, also follows by Theorem \ref{thm:mainthm}. Given a spherical function over $\Q(q)$, one may extend the spherical function to a spherical function for the quantum symmetric pair over $\F$. Here, Weyl group invariance follows by Theorem \ref{thm:mainthm}. Afterwards, one may restrict the spherical functions back to the $\Q(q)$-form of the quantum group. In this way we obtain Weyl group invariance for spherical functions for quantum symmetric pairs with balanced parameters over $\Q(q)$. 
\end{rem}
\subsection{A comparison with Letzter's invariance}
In this section we compare Theorem \ref{thm:mainthm} to Gail Letzter's Weyl group invariance \cite[\mbox{Thm 5.3}]{Letzter2003}. 
For the trivial character, Gail Letzter shows that for each coideal subalgebra $\uqb$ there exists a coideal subalgebra $\uqd$ such that the $(\uqd,\uqb)$ zonal spherical functions are Weyl group invariant. In our notation $\uqd$ can be expressed as $\uqd=\mathrm{ad}(K_\rho)(\uqb)$. Let $\Res(\uq^0)$ denote the restriction map $\Res(\uq^0): \qfa\to \F[X]$ , cf. \cite[\mbox{Sec 4}]{Letzter2003}. Note that the relative Weyl group $W^\Sigma$ naturally acts on $\F[X]$.
\begin{thm}\cite[\mbox{Thm 5.3}]{Letzter2003}.
	Restriction to the quantum torus defines a $\F$-linear map
	$\mathrm{Res}(\uq^0): {^{\mathrm{ad}(K_\rho)(\uqb)}_{\qquad\quad\,\,\,\epsilon}\qfa^{\uqb}_\epsilon} \to \F[X]^{W^\Sigma}$
\end{thm}
The goal of this section is to understand the relation between $\mathrm{ad}(K_\rho)(\uqb)$ and $\uqb^\chi$. To start, we specialize Theorem \ref{thm:mainthm} to the trivial character.
\begin{cor}\label{cor:Letzter}
	Let $\boldsymbol{c}$ be a balanced parameter. Then restriction to the quantum torus defines a $\F$-linear map
	$\mathrm{Res}(\uq^0): {^{\uqb^\epsilon}_{\epsilon}\qfa^{\uqb}_\epsilon} \to \F[X]^{W^\Sigma}.$
\end{cor}
\begin{proof}
Let $h\in Y$ with $h=h_1+h_2$, where $\Theta(h_1)=h_1$ and $\Theta(h_2)=-h_2$. Let $\varphi \in {^{\uqb^\epsilon}_{\epsilon}\qfa^{\uqb}_\epsilon}$, then \begin{equation}\label{eq:trans}
\varphi(K_{h_1+h_2})=\epsilon(K_{h_2})\varphi(K_{h_2})=\varphi(K_{h_2}).
\end{equation}
The action of $W^\Sigma$ on $\mathcal A$ and $\uq_\Theta^0$ preserves these spaces, respectively. Using (\ref{eq:trans}) we see that to have Weyl group invariance, it suffices to be Weyl group invariant on the quantum torus. As $\epsilon|_{\uqb^\epsilon}$ is akin to $\epsilon|_{\uqb}$ Weyl group invariance on the quantum torus follows by Theorem \ref{thm:mainthm}.
\end{proof}
\begin{rem}
	For zonal spherical functions, the assumption that the parameter is balanced may be relaxed to more general non-standard, non-balanced parameters that are as described in Remark \ref{rem:asumpar} using Theorem \ref{thm:weylinvrank1}.
\end{rem}
Using \cite[\mbox{Lem 3.10}]{Bao2018}, one deduces the values of the parameter given in Proposition \ref{prop:finaly}. The parameter in the last column is the parameter shift as given in  \cite[\mbox{Thm 5.3}]{Letzter2003}. For $i\in \I$ recall the positive integers $d_i$, introduced in the beginning of Section \ref{sec:quantumgrouips}.
\begin{center}
	\captionof{table}{Values of parameters for quantum symmetric pairs of rank one}\label{table}
	\begin{tabular}{|c|c|c|c|c|}
		\hline
		\textbf{Type}	& $d_i\langle\alpha_i^\vee,\alpha_i\rangle$  & ${\langle 2
			\rho_{\bullet}^\vee,\alpha_i\rangle } $ &$-d_i\langle 2\rho_\bullet,\alpha_i^\vee\rangle$&$ (\rho,\alpha_i-\Theta(\alpha_i))$   \\
		\hline
		$\mathsf{AI_1}$&2  &0  & 0 & 2   \\
		\hline
		$\mathsf{AII_3}$	&2  &-2  & 2 &4    \\
		\hline
		$\mathsf{AIII_{11}}$		&2  &0  & 0 & 2   \\
		\hline
		$\mathsf{AIV}$, $n\geq 2$	&$2$  &$2-n$  &$n-2$  & $n$    \\
		\hline
		$\mathsf{BII}$, $n\geq 2$	&4  &$-2n+2$  & $2(2n-3)$  & $2(2n-1)$    \\
		\hline
		$\mathsf{CII}$, $n\geq 3$	&$2$  &  $-2n+4$&$2n-3$  & $2n-1$    \\
		\hline
		$\mathsf{DII}$, $n\geq 4$	&$2$  & $-2n+4$ &$2n-4$  & $2n-2$   \\
		\hline
		$\mathsf{FII}$	&$2$  & $-6$ & $9$& $11$   \\
		\hline
	\end{tabular}
\end{center}
Using Table \ref{table} we observe that the parameters
$$q^{(\rho,\alpha_i-\Theta(\alpha_i))}c_i\quad\text{and}\quad  (-1)^{\langle 2\rho_{\bullet}^\vee,\alpha_i\rangle}q_i^{\langle\alpha_i^\vee,\alpha_i\rangle}
q_i^{-\langle \alpha_i^\vee,2\rho_\bullet\rangle}c_i,\qquad\text{for}\qquad i\in\I_\circ$$
coincide up to the scalar $(-1)^{\langle 2\rho_{\bullet}^\vee,\alpha_i\rangle}$. This scalar is not equal to one in the case of $\mathsf{AIV}$ with $n\geq 2$ odd. In the following proposition we show how to gauge the parameter.
\begin{prop}
	Let $(\uq,\uqb)$ be a quantum symmetric pair, where the Satake diagram is of rank one and of type $\mathsf{AIV}$ with $n\geq 2$ odd. If the $(\mathrm{ad}(K_{\rho})(\uqb),\uqbmin)$\\ zonal spherical functions are Weyl group invariant, then the $(\mathrm{ad}(K_{\rho})(\uqb),\uqb)$ zonal spherical functions are Weyl group invariant.
\end{prop}
\begin{proof}
	Let $c_{f,v}\in L^r(\lambda)\otimes L(\lambda)$ be a zonal spherical function and let $\boldsymbol{a}\in (\F^\times)^\I$ be such that $a_j=1$ for $j\in\I_\bullet$ and $a_1=(-1)^{1/2}=a_n$. By Lemma \ref{lem:tech4},  $\Phi_{\boldsymbol{a}}(v)$ is a zonal spherical vector for $\uqb$. We may write $v=\sum_{J\in \mathcal I}b_jE_Jv_{-w_0\lambda}$, where for each $J\in\mathcal I$ we have $E_J\in \uq_{n (\alpha_1+w_\bullet a_n)}^+$ for some $n\in \N$ and $b_j\in \F$. 
	Let $E_J\in \uq_{n (\alpha_1+w_\bullet a_n)}^+$. Then we have
	\begin{align}\label{eq:properties}K_{i}K_{w_\bullet \tau(i)}E_J=q^{2n}E_JK_{i}K_{w_\bullet \tau(i)}\qquad\text{and}\qquad \Phi_{\boldsymbol{a}}(E_J)=(-1)^{n}E_J.\end{align}
	By \cite[\mbox{Sec 4.4}]{watanabe2022stability}, there exists a $m\geq0$ such that $\lambda=m(\omega_1+\omega_n)$, and the quantum torus is generated by $K_{i}K_{w_\bullet \tau(i)}$. Hence, as a result of Theorem \ref{thm:mainthm}, there exist scalars $a_k\in \F$ such that
	$
	c_{f,v}((K_{i}K_{w_\bullet \tau(i)})^n)=\sum_{k=0}^m a_k(q^{2nk}+q^{-2nk}).
	$
	Using (\ref{eq:properties}), we obtain
	$
	c_{f,{\Phi}_{\boldsymbol{a}}(v)}((K_{i}K_{w_\bullet \tau(i)})^n)=\sum_{k=0}^m a_k(-1)^{k-m}(q^{2nk}+q^{-2nk}).
	$
	Because $K_{i}K_{w_\bullet \tau(i)}$ generates the quantum torus, this means that $c_{f,{\Phi}_{\boldsymbol{a}}(v)}\in\qfaaws$.
\end{proof}
\begin{rem}
	Assume that $c_i=-q_i^{-1}$ if $i\in \I_{\mathrm{ns}}$.
	Using Table \ref{table} and the \textit{shift of basepoint}, cf. \cite[\mbox{Eq (10.1)}]{Kolb2023}, defined by
	$$\rho_\chi: \uqb\to \rho_\chi(\uqb),\qquad b\mapsto \chi(b_{(1)})b_{(2)},\quad \text{for}\quad b\in \uqb,$$
	the coideal subalgebra $\uqb^\chi$ can be described as
	$\uqb^\chi=\Phi_{\boldsymbol{s}}\big(\mathrm{ad}(K_\rho)(\rho_\chi(\uqb))\big)$.
\end{rem}
\subsection{Examples}
This section contains explicit examples of spherical functions of type $\chi$.
The following is a special case of the representation theory studied in \cite{Aldenhoven2017} and \cite{Watanabe2018}.
\begin{ex}[Spherical functions for non-trivial characters]
	Let $\mathbf{V}$ be the vector representation of $\mathbf U_q(\mathfrak{sl}_3)$, with standard basis $\{v_1,v_2,v_3\}$. Let $\tau:\{1,2\}\to \{1,2\}$ be the diagram automorphism with $\tau(1)=2$ and $\tau(2)=1$. Then $(\{1,2\},\emptyset,\tau)$ is a Satake diagram. 
	Let $\boldsymbol{c}$ be a parameter and let $\uqb$ to be the corresponding coideal subalgebra. The vector $v= v_1-c_1^{-1}v_3\in \mathbf{V}$ spans a one-dimensional $\uqb$ submodule that is not isomorphic to the trivial module and $v^r=v_1-q^{-1} c_2v_3$ spans a one-dimensional submodule for $\uqb$ in $^\varrho \mathbf{V}$. With respect to the Shapovalov form the spherical function $c_{v^r,v}$ takes the form
	$$c_{v^r,v}(K_{n\rho})=q^{-1}c_1^{-1}c_2q^{-n}+q^n\qquad \text{for}\qquad n\in \Z.$$
	It holds that $q^{(\rho,\alpha_{\tau(i)}+\alpha_i)}\chi(K_1^{-1}K_2)=q^{2}q^{-1}=q.$
	Thus, by definition of $\uqb^\chi$ the spherical function $c_{\overline{v},v}$ for $(\uqb^\chi,\uqb)$ takes the form
	$$c_{v^r,v}(K_{n\rho})=c_{\overline{v},v}(K_{n\rho})=q^{-n}+q^n\qquad \text{for}\qquad n\in \Z,$$
	which one observes to be relative Weyl group invariant, cf. Theorem \ref{thm:mainthm}. 
\end{ex}
\begin{ex}
	Let $L(\omega_2)=\bigwedge_q^2 L(\omega_1)$ be the highest weight module of weight $\omega_2$ of $\mathbf U_q(\mathfrak{sl}_4)$, see \cite[\mbox{4.3.12}]{Stok98} for the definition of $\bigwedge_q^2 L(\omega_1)$. Consider the rank $2$ Satake diagram of type $\mathsf{AIII_{3}}$.  Let $\boldsymbol{c}$ be the uniform balanced parameter defined by $c_1=1$ and $c_2=q^{-1}$. Let $\uqb$ be the corresponding quantum symmetric pair coideal subalgebra. 
	The vector representation $\mathbf{V}\cong L(\omega_1)$ has a standard basis $\{v_1,v_2,v_3,v_4\}$. The vector space $\bigwedge_q^2 L(\omega_1)$ has a basis $\{v_i\wedge v_j\,:\, 1\leq i<j\leq 4\}$.
	The vector $v= v_1\wedge v_2-v_1\wedge v_3+q^{-1}v_2\wedge v_4-q^{-1}v_3\wedge v_4\in L(\omega_2)$ spans a one-dimensional $\uqb$ submodule that is not isomorphic to the trivial module. Let $\chi\in \widehat{\uqb}$ be the corresponding character. The non-standard quantum symmetric pair coideal subalgebra $\uqb^\chi$ is generated by $\uq_{\Theta}^0$ and
	$$B_1=F_1+q^{2}E_3K_1^{-1},\quad B_3=F_3+q^{2}E_1K_3^{-1}\quad\text{and}\quad B_2=F_2+qE_2K_2^{-1}-qK_2^{-1}.$$
	One may check that the vector $\overline{v}\in L^r(\omega_2)$ spans a one-dimensional $\uqb^\chi$ module akin to $\chi$. The spherical function $c_{\overline{v},v}$ takes the following form on the quantum torus:
	$$c_{\overline{v},v}(K_{n\alpha_2+m(\alpha_1+\alpha_3})=q^n+q^{2m-n}+q^{n-2m}+q^{-n},\qquad \text{for}\qquad n,m\in \Z.$$
	We have $\boldsymbol{r_1}=s_1s_3$ and $\boldsymbol{r_2}=s_2$, and their actions on $\alpha_2$ and $\alpha_1+\alpha_3$ are given by
	$$\boldsymbol{r_2}(\alpha_1+\alpha_3)=\alpha_1+\alpha_3+2\alpha_2\qquad\text{and}\qquad \boldsymbol{r_1}(\alpha_2)=\alpha_1+\alpha_2+\alpha_3.$$
	Therefore, for all $n,m\in\N$ we have
	\begin{align*}
		c_{\overline{v},v}(\boldsymbol{r_1}K_{n\alpha_2+m(\alpha_1+\alpha_3)})&=c_{\overline{v},v}(K_{n\alpha_2+(n-m)(\alpha_1+\alpha_3)})\\
		&=q^n+q^{-n}+q^{2(n-m)-n}+q^{-(2(n-m)-n)}\\
		&=q^n+q^{2m-n}+q^{n-2m}+q^{-n}\\
		&=c_{\overline{v},v}(K_{n\alpha_2+m(\alpha_1+\alpha_3)}).
	\end{align*}
	and
	\begin{align*}
		c_{\overline{v},v}(\boldsymbol{r_2}K_{n\alpha_2+m(\alpha_1+\alpha_3)})&=c_{\overline{v},v}(K_{(2m-n)\alpha_2+m(\alpha_1+\alpha_3)})\\
		&=q^{2m-n}+q^{n-2m}+q^{2m-(2m-n)}+q^{n-2m}+q^{-(2m-(2m-n))}\\
		&=q^n+q^{2m-n}+q^{n-2m}+q^{-n}\\
		&=c_{\overline{v},v}(K_{n\alpha_2+m(\alpha_1+\alpha_3)}).
	\end{align*}
	Hence, we observe Weyl group invariance of $c_{\overline{v},v}$
\end{ex}
\begin{ex}(Bar invariance of zonal spherical functions)
	Let $(\uq,\uqb)$ be a quasi-split quantum symmetric pair, meaning that $\I_\bullet=\emptyset$. Furthermore, suppose that $\tau_0=\tau$ and that $\boldsymbol{c}$ is a uniform balanced parameter. Because $(\uq,\uqb)$ is quasi-split, \cite[\mbox{Prop 4.1.3}]{Watanabe2023} tells us that the spherical weights are $\tau$ invariant. Hence, if $\varphi\in L(\lambda)^\ast\otimes L(\lambda)$ is a zonal spherical function, then $\varphi\circ\tau$ is a matrix entry in $L(\lambda)^\ast \otimes L(\lambda)$. Using that $\boldsymbol{c}$ is balanced one shows that $\varphi\circ\tau$ is a spherical function. By uniqueness of zonal spherical functions in $L(\lambda)^\ast \otimes L(\lambda)$, see Theorem \ref{thm:unique2}, there exists $\eta\in\F$ such that $\varphi\circ \tau=\eta\varphi$. We obtain $\eta=1$, by evaluating at $1\in \uq$. Proposition \ref{prop:invbar} implies that
	$$\varphi(K_h)=\overline{\varphi(K_{\tau_0(h)})}=\overline{\varphi(K_{\tau(h)})}=\overline{\varphi(K_{h})}\qquad\qquad h\in Y,$$
	which shows that zonal spherical functions in this setting are bar invariant on $\uq^0$.
\end{ex}
\begin{appendices}
	\section{Proof of Proposition \ref{prop:charinv}}\label{Appendix}
	\addtocontents{toc}{\protect\setcounter{tocdepth}{1}}
	In this appendix we will conclude the proof of Proposition \ref{prop:charinv}. We present a constructive proof that for each character $\chi\in \widehat{\uqdis}$ that occurs in finite dimensional $\uq$ weight modules the equality $\chi\circ \boldsymbol{T}_{i,-1}'(B_j)=\chi(B_j)$ holds. We work over the Drinfel'd double $\widetilde{\uq}$, see \cite[\mbox{Section 2.1}]{Wang2023}, because there the actions are explicitly described in \cite[\mbox{Table 3}]{Wang2023}. We use the notation $\mathcal K_i=K_iK_{w_\bullet \tau(i)}'$, for $i\in \I_\circ$. 
	\begin{lem}\label{lem:aprang1}
		For each $i\in\I_\circ$ it holds that $\chi\circ\boldsymbol{T}_{i,-1}'(B_i)=\chi(B_i)$.
	\end{lem}
	\begin{proof}
		By Lemma \ref{lem:tech3}, Lemma \ref{lem:unique} and Proposition \cite[\mbox{Prop 10.1}]{Kolb2023} it remains to check the statement in the case that the rank one Satake diagrams induced by $i\in \I_\circ$ is of type $\mathsf{AI}$. We deduce from \cite[\mbox{Thm 4.14}]{Wang2023} that 
		$$\boldsymbol{\widetilde{T}}'_{i,-1}(B_i)=-q_i^{-2}B_i\mathcal K_i^{-1}.$$
		As $\mathcal K_i^{-1}=K_i'K_i$ is mapped to $-q_i^{2}$ under the central reduction, see \cite[\mbox{Table 1}]{Wang2023}, it follows that 
		$\boldsymbol{T}'_{i,-1}(B_i)=B_i.$ In particular we have $\chi\circ \boldsymbol{T}'_{i,-1}(B_i)=\chi(B_i.)$
	\end{proof}
	By Lemma \ref{lem:aprang1}, it suffices to check that 
	\begin{equation}\label{eq:chisym}
		\chi\circ \boldsymbol{T}_{i,-1}'(B_j)=\chi(B_j)
	\end{equation}
	where $i\neq j\in \I_\circ$ lie in a rank two Satake diagram and $j\in \I_{\mathrm{ns}}$ where there exist non-trivial characters with $\chi(B_j)\neq0$. This rules out the Satake diagrams where there exist subdiagrams of type $\mathsf{AIV}$ with $\I_\bullet\neq \emptyset$. Hence using Proposition \ref{cor:herm2} it remains to check (\ref{eq:chisym}) for the rank two diagrams where $(\mathfrak{g},\mathfrak{k})$ is of Hermitian type and there exists $i\in \I_{\mathrm{ns}}$. By $\cite[\mbox{X. \S6.3}]{Helga98}$ these are the diagrams of type $ \mathsf{AIII}_3,$ $\mathsf{BI}_n,$ $\mathsf{CI}_2$, $\mathsf{DI}_n,$ $\mathsf{DIII}_4.$
	We label the roots as in \cite[\mbox{Table 3}]{Wang2023}.
	\subsection{Type $\boldmath{\mathsf{AIII}_3}$}\label{subsec:AI}
	\addtocontents{toc}{\protect\setcounter{tocdepth}{1}}
	We deduce from \cite[\mbox{Table 3}]{Wang2023} that $$\boldsymbol{\widetilde{T}}'_{1,-1}(B_2)=[B_3,[B_1,B_2]_q]_q-q B_2 \mathcal K_3.$$
	In this case $\mathcal K_3=K_3K_{\tau(3)}'$, which is mapped to $-q^{-1}$ under the central reduction \cite[\mbox{Table 1}]{Wang2023}. It follows that 
	$$\boldsymbol{T}'_{1,-1}(B_2)=[B_3,[B_1,B_2]_q]_2+ B_2 .$$
	As $\chi(B_1)=\chi(B_3)=0$, it follows that
	$\chi (\boldsymbol{T}'_{1,-1}(B_2))= \chi(B_2).$
	\subsection{Type $\boldmath{\mathsf{BI}_n}$}
	We see in \cite[\mbox{Table 3}]{Wang2023} that $$\boldsymbol{\widetilde{T}}'_{2,-1}(B_1)=[\widetilde{\mathcal T}_{w_\bullet}(B_2),[B_2,B_1]_{q_2}]_{q_2}-q_2 B_1 \widetilde{\mathcal T}_{w_\bullet}(\mathcal K_2).$$
	In this case $\widetilde{\mathcal T}_{w_\bullet}(\mathcal K_2)$ is mapped to $-q_2^{-1}$ under the central reduction \cite[\mbox{Table 1}]{Wang2023}. As $\chi(B_2)=0$, it follows that 
	$\chi (\boldsymbol{T}'_{2,-1}(B_1))= \chi(B_1).$
	\subsection{Type $\boldmath{\mathsf{CI}_2}$}
We see in \cite[\mbox{Table 3}]{Wang2023} that $$\boldsymbol{\widetilde{T}}'_{1,-1}(B_2)=\frac{1}{[2]_{q_1}}[B_1,[B_1,B_2]_{q_1^2}]-q_1^2B_2\mathcal K_1.$$
	In this case $\mathcal K_1$ is mapped to $-q_1^{-2}$ under the central reduction \cite[\mbox{Table 1}]{Wang2023}. Thus it follows that 
	$\chi (\boldsymbol{\widetilde{T}}'_{1,-1}(B_2))= \chi(\frac{1}{[2]_{q_1}}[B_1,[B_1,B_2]_{q_1^2}]+B_2)=\chi(B_2).$
	\subsection{Type $\boldmath{\mathsf{DI}_n, \, n\geq 5}$} 
	We see in \cite[\mbox{Table 3}]{Wang2023} that $$\boldsymbol{\widetilde{T}}'_{2,-1}(B_1)=[\widetilde{\mathcal {T}}_{w_\bullet}(B_2),[B_2,B_1]_q]_q-q B_1  \widetilde{\mathcal T}_{w_\bullet}(K_2).$$
	Here $\ \widetilde{\mathcal T}_{w_\bullet}(K_2)$ is mapped to $-q^{-1}$ under the central reduction, see \cite[\mbox{Table 1}]{Wang2023}. As $\chi(B_2)=0$, it follows that
	$\chi (\boldsymbol{T}'_{2,-1}(B_1))= \chi(B_1).$
	\subsection{Type $\boldmath{\mathsf{DIII}_4}$ }\label{sub:DIII} 
	We see in \cite[\mbox{Table 3}]{Wang2023} that $$\boldsymbol{\widetilde{T}}'_{2,-1}(B_1)=[\widetilde{\mathcal {T}}_{w_\bullet}(B_2),[B_2,B_1]_q]_q-q B_1  \widetilde{\mathcal T}_{w_\bullet}(K_2).$$
	Here,$\ \widetilde{\mathcal T}_{w_\bullet}(K_2)$ is mapped to $-q^{-1}$ under the central reduction, see \cite[\mbox{Table 1}]{Wang2023}. As $\chi(B_2)=0$, it follows that that
	$\chi (\boldsymbol{T}'_{2,-1}(B_1))= \chi(B_1).$
	\section{Proof of Lemma \ref{lem:lusztiglong}}\label{apenB}
	This appendix contains the proof of Lemma \ref{lem:lusztiglong}.
	We restrict to the case that $\I_\bullet \neq \emptyset$, because the statement is trivial otherwise. For each $i \in \I_\circ$, the elements $T_{w_\bullet,\pm1}''(E_i)$ only depend on the rank one Satake subdiagram induced by $i$. Henceforth, we show that in each of the rank one Satake diagrams it holds that $T_{w_\bullet,+1}''(E_i)v=T_{w_\bullet,-1}''(E_i)v$. This is done by a case-by-case analysis, since we were not able to find a uniform proof. The labeling of the roots is as in \cite[\mbox{Table 1}]{Wang2023}. Recall from \cite[\mbox{37.1.3}]{Lusztig2010}, that for $k ,j\in\I$ with $k\neq j$ it holds that
	\begin{align}\label{eq:lusztig}
		T''_{k,e}(E_j)=\sum _{r+s=-\langle \alpha_k^\vee,\alpha_j\rangle}(-1)q_k^{-er}E_k^{(s)}E_jE_k^{(r)}.
	\end{align}
	\subsection{Type $\mathsf{AII_3}$}
	For type $\mathsf{AII_3}$, we have $w_\bullet=s_1s_3$. Let $e=\pm1$. Using Lemma \ref{lem:uqbullet} and (\ref{eq:lusztig}) we have
	$$T''_{w_\bullet,e}(E_2)v=T''_{1,e}T''_{3,e}(E_2)v=T''_{1,e}(E_3E_2)v=E_3E_1E_2v.$$
	\subsection{Type $\mathsf{AIV}$, $n\geq 2$}
	Using \cite[\mbox{Table 1}]{Benkart2012}, there exists a reduced expression of $w_\bullet $ given by
	$$w_{\bullet }=(s_{n-1}s_{n-2}\dots s_{2})(s_{n-1}\dots s_{3} )( s_{n-1}\dots s_{4})\dots (s_{n-1} s_{n-2}).$$
	We remark that $T_{(s_{n-1}\dots s_{3} )( s_{n-1}\dots s_{4})\dots (s_{n-1} s_{n-2}),e}''(E_1)=E_1$. Thus, using Lemma \ref{lem:uqbullet} and (\ref{eq:lusztig}), one has
	\begin{align*}
		T_{w_{\bullet} ,e}''(E_1)v&=T_{s_{n-1}\dots s_3,e}''\circ T''_{{s_2},e}(E_1)v\\
		&=\sum _{r+s=-\langle \alpha_{2}^\vee,\alpha_1\rangle}T_{s_{n-1}\dots s_3,e}''\big((-1)^rq^{-er}E_{2}^{(s)}E_1E_{2}^{(r)}\big)v\\
		&=T_{s_{n-1}\dots s_3,e}''\big(E_{2}E_1\big)v\\
		&=T_{s_{n-1}\dots s_3,e}''\big(E_{2}\big)E_1v\\
		&=\sum _{r+s=-\langle \alpha_{3}^\vee,\alpha_2\rangle}T_{s_{n-1}\dots s_4,e}''\big((-1)^rq^{-er}E_{3}^{(s)}E_2E_{3}^{(r)}\big)E_1v\\ 
		&=T_{s_{n-1}\dots s_4,e}''\big(E_{3}\big)E_2E_1v\\
		&=\dots\\
		&=E_{n-1}\dots E_2E_1v,
	\end{align*}
	which is independent of $e$. The proof for $T_{w_{\bullet} ,e}''(E_n)v$ is analogous.
	\subsection{Type $\mathsf{BII}$, $\mathsf{CII}$ or $\mathsf{DII}$}
	We prove the statement for the Satake diagram of type $\mathsf{BII}$. The proofs for the types $\mathsf{CII}$ and $\mathsf{DII}$ are analogous. For $m\leq n$ define $\I_m:=\{m,\dots n\}$. We write $w_{\I_m}$ to be the corresponding longest element, the corresponding diagram involution $\varsigma_m:\I_m\to \I_m$ equals the identity. Using \cite[\mbox{Table 1}]{Benkart2012}, a reduced expression of $w_\bullet$ is given by $w_\bullet = s_2\dots s_n\dots s_2w_{\I_3}$. We define $w'=s_2w_\bullet$.  Because $w_\bullet \alpha_{2}=- \alpha_{2}$, we have $w'\alpha_{2}=\alpha_{2}$. Using Lemma \ref{lem:uqbullet}, (\ref{eq:lusztig}) and \cite[\mbox{Prop 8.20}]{Jantzen1996}, we have {\allowdisplaybreaks
		\begin{align*}
			T_{w_{\bullet} ,e}''(E_1)v&=T_{w',e}''\circ T''_{{2},e}(E_1)v\\
			&=T_{w',e}''(E_{2}E_1)v\\
			&=E_{2}T_{w',e}''(E_1)v\\
			&=E_{2}T_{w_{\I_3},e}''(E_2)E_1v\\
			&=E_{2}E_3T_{w_{\I_4},e}''(E_3)E_2E_1v\\
			&=\dots\\
			&=E_{2}E_3\dots E_{n-1}E_n^{(2)}E_{n-1}\dots E_2E_1v.
	\end{align*}}
	\subsection{Type $\mathsf{FII}$}
	The longest word $w_\bullet$ is the longest word of the Weyl group of type $\mathsf{B_3}$. By \cite[\mbox{Table 1}]{Benkart2012}, a reduced expression of $w_\bullet$ is given by $s_1s_2s_3s_2s_1s_2s_3s_2s_3$. Remark that
	$$s_1s_2s_3s_2s_1(\alpha_2)=\alpha_2\qquad\text{and}\qquad s_1s_2s_3s_2s_1(\alpha_3)=\alpha_3.$$
	Let $e=\pm1$. Using \cite[\mbox{Prop 8.20}]{Jantzen1996}, we have{\allowdisplaybreaks
		\begin{align*}
			T_{w_\bullet,e}''(E_4)v&=T_{w_\bullet s_3^{-1},e}''(E_3E_4)v\\
			&=E_3T_{w_\bullet s_3^{-1},e}''(E_4)v\\
			&=E_3T_{s_1s_2s_3s_2s_1s_2,e}''(E_3E_4)v\\
			&=E_3T_{s_1s_2s_3s_2s_1s_2,e}''(E_3)T_{s_1s_2s_3s_2s_1s_2,e}''(E_4)v\\
			&=E_3T_{s_1s_2s_3s_2s_1s_2,e}''(E_3)E_1E_2E_3E_4v\\
			&=E_3T_{s_1s_2s_3s_2s_1,e}''(E_2)T_{s_1s_2s_3s_2s_1,e}''(E_3)E_1E_2E_3E_4v\\
			&=E_3E_2E_3E_1E_2E_3E_4v.
	\end{align*}}
\end{appendices}
\begin{center}
	$\textsc{Conflict of Interest}$
\end{center}
The author has no conflict of interest to declare that are relevant to this article.
\begin{center}
	$\textsc{Data Availability}$
\end{center}
 Data sharing not applicable to this article as no data sets were generated or analysed during the current study.
\bibliographystyle{alpha}
\bibliography{Reff}
\addcontentsline{toc}{section}{References}
\end{document}